\documentclass{llncs}
\usepackage{amssymb}
\usepackage[english]{babel}
\usepackage{mathrsfs}
\usepackage[colorlinks, urlcolor=blue, citecolor=black, linkcolor=black, breaklinks]{hyperref}
\usepackage{algorithm}
\usepackage{algpseudocode}
\usepackage{tikz-cd}
\usepackage[all]{xy}

\algnewcommand{\LineComment}[1]{\State \(\triangleright\) #1}
\usepackage{mathtools}
\usepackage{etoolbox}
\usepackage{caption}
\usepackage{appendix}
\usepackage{comment}
\usepackage{listings}
\usepackage{color} 
\definecolor{mygreen}{RGB}{28,172,0}
\definecolor{mylilas}{RGB}{170,55,241}
\newcommand{\F}{{\mathbb F}}

\def\og{\leavevmode\raise.3ex\hbox{$\scriptscriptstyle\langle\!\langle$~}}
\def\fg{\leavevmode\raise.3ex\hbox{~$\!\scriptscriptstyle\,\rangle\!\rangle$}}
\usetikzlibrary{calc,backgrounds}
\def\F{\mathbb{F}}

\newcommand{\Ld}[2][D]{\ensuremath{\mathcal{L}(#2\mathcal{#1})}} 
\newcommand{\D}[1][D]{\ensuremath{\mathcal{#1}}}

\begin{document}

\title{Multiplication in finite fields with Chudnovsky-type algorithms on the projective line}

\author{St\'ephane Ballet \and Alexis Bonnecaze \and Bastien Pacifico} 
\institute{Aix Marseille Univ, CNRS, Centrale Marseille, I2M, Marseille,  France
\email{Stephane.Ballet@univ-amu.fr\\Alexis.Bonnecaze@univ-amu.fr\\Bastien.Pacifico@univ-amu.fr}}

\maketitle

\begin{abstract}
We propose a Recursive Polynomial Generic Construction (RPGC) of
multiplication algorithms in any finite field $\mathbb{F}_{q^n}$ based
on the method of D.V. and G.V. Chudnovsky specialized on the
projective line. They are usual polynomial interpolation algorithms in
small extensions and the Karatsuba algorithm is seen as a particular
case of this construction. Using an explicit family of such
algorithms, we show that the bilinear complexity is quasi-linear with
respect to the extension degree $n$, and we give a uniform bound for
this complexity. We also prove that the construction of these
algorithms is deterministic and can be done in polynomial time. We
give an asymptotic bound for the complexity of their construction.
\end{abstract}


%\ccode{2020 Mathematics Subject Classification: 12-08, 14Q05, 14Q20}

%	12-08  	Computational methods for problems pertaining to field theory
%	14Q05  	Computational aspects of algebraic curves
%	14Q20  	Effectivity, complexity and computational aspects of algebraic geometry

\section{Introduction}

Multiplication in finite fields has been at the heart of many works since the end of the twentieth century. In addition to being interesting for the theoretical side, this subject is also very current for its applications in computer science, such as in information theory. Different strategies have been studied to build a multiplication algorithm. Among them, interpolation algorithms on algebraic curves, due to D.V. and G.V. Chudnovsky \cite{chch}, have been widely studied for their qualities in terms of bilinear complexity \cite{survey}. Nevertheless, they present a certain number of weaknesses, including their difficulty of construction and use. In this paper, we propose a construction method that allows us to bypass these difficulties, by doing polynomial interpolation, while preserving the benefit of Chudnovsky-type algorithms.

Multiplications in a degree $n$ extension of $\mathbb{F}_q$ require different kind of operations in $\mathbb{F}_q$.
Let $x=\sum_{i=1}^{n}x_ie_i$ and $y=\sum_{i=1}^{n}y_ie_i$ be two elements of $\mathbb{F}_{q^n}$, in a basis $\{e_1,\ldots,e_n\}$ of $\mathbb{F}_{q^n}$ over $\mathbb{F}_q$. By the usual method, the product of $x$ and $y$ is given by the formula 
\begin{equation}\label{calculdirect}
z=xy=\sum_{h=1}^{n}z_he_h=\sum_{h=1}^{n}\biggr( \sum_{i,j=1}^{n}t_{ijh}x_iy_j\biggl)e_h,
\end{equation}
 with $$e_ie_j=\sum_{h=1}^{n}t_{ijh}e_h,$$ 
where $t_{ijh}\in \mathbb{F}_q$ are constants in $\mathbb{F}_q$. Two different types of multiplications are involved in this product. The scalar ones are multiplications by a constant in $\mathbb{F}_q$, and the bilinear ones depend on the two elements being multiplied (i.e. the $x_iy_j$). Of the two, bilinear multiplications are known to be computationally heavier (\cite{shtsvl}, see Survey \cite{survey}). This explains the motivation to reduce the number of bilinear multiplications in multiplication algorithms and led to the study of the bilinear complexity, that can be defined as follows.

\begin{definition}
Let $\mathcal U $ be an algorithm for the multiplication in $\mathbb F_{q^n}$ over $\mathbb{F}_q$. Its number of bilinear multiplications is called its bilinear complexity, written $\mu(\mathcal{U})$.
The bilinear complexity of the multiplication in $\mathbb{F}_{q^n}$ over $\mathbb{F}_q$, denoted by $\mu_q(n)$, is the quantity:
$$\mu_q(n)=\underset{\mathcal U}{\min}~\mu(\mathcal{U}),$$
where $\mathcal{U}$ is running over all multiplication algorithms in $\mathbb{F}_{q^n}$ over $\mathbb{F}_q$.
\end{definition}

\subsection{Some known-results}

From the works of Winograd and De Groote \cite{groo} applied to the multiplication in any finite field $\mathbb{F}_{q^n}$, it is proven that for all $n$ we have $\mu_q(n)\geq2n-1$,  equality being ensured if and only if $n\leq\frac{1}{2}q+1$ (\cite{survey}, Theorem 2.2). Winograd also proved that this lower bound is obtained with interpolation algorithms \cite{wino3}. 
In 1987, D.V. and G.V. Chudnovsky proposed an interpolation method on algebraic curves \cite{chch}, generalizing polynomial interpolation. This method makes it possible to multiply in any extension of degree $n$ of $\mathbb{F}_q$, with a good bilinear complexity, provided that one has an algebraic curve with a sufficient number of rational points. This original algorithm is called the Chudnovsky-Chudnovsky Multiplication Algorithm (CCMA). More generally, a multiplication algorithm using interpolation over algebraic curves is said to be of type Chudnovsky.

For an increasing degree of the extension, the interpolation requires more and more rational points (i.e. rational places). From the Serre-Weil bound, the number of rational places is bounded for a fixed genus. Hence, the classical strategy is to build these algorithms over function fields of growing genus. Ballet proved that the bilinear complexity is linear in the degree of the extension (\cite{ball1}, see \cite{survey}) using the original algorithm over an explicit tower of function fields defined by Garcia and Stichtenoth \cite{gast}. However, it is not clear that these algorithms can be constructed in a reasonable time since we have no method to find the place of degree $n$ required to represent $\mathbb{F}_{q^n}$ (\cite{shtsvl}, Remark 5). Moreover, there is no generic and deterministic construction for both the divisors and the basis of the Riemann-Roch spaces involved in the algorithms. 

The strategy of growing genus was natural since the original algorithm evaluates only on rational places of a function field. But several years later, thanks to the works of Ballet and Rolland \cite{baro1}, Arnaud \cite{arna1}, Cenk and Özbudak \cite{ceoz}, and Randriambololona \cite{randJComp}, the method has been extended to the evaluations at places of higher degrees, and to the use of derivative evaluations. These generalizations led to the introduction of another strategy for constructing the algorithm for asymptotically large extensions. The evaluation at places of higher degrees allows one to fix a function field and to evaluate at places of growing degrees. In \cite{babotu}, Ballet, Bonnecaze and Tukumuli built Chudnovsky-type algorithms with interpolation only over elliptic curves, i.e. fixing the genus $g$ of the function field to be equal to 1, and using places of increasing degrees. This work gave a quasi-linear asymptotic bound for the bilinear complexity of these algorithms with respect to the degree of the extension. Moreover, they can be constructed in polynomial time. This latest result is not yet established for the growing genus strategy.

\subsection{New results and organization}

In this paper, we build Chudnovsky-type algorithms for the multiplication in any finite field $\mathbb{F}_{q^n}$, with interpolation only over the projective line, i.e. fixing the genus $g$ to be equal to 0, and using places of increasing degrees. In small extensions, the bilinear complexity of the obtained algorithms can equalize the best known bound and sometimes even improve it. Compared with the construction over elliptic curves, our work has the advantages of giving an uniform bound for the bilinear complexity of our algorithms and giving a generic construction of algorithms for the multiplication in any finite field. Namely, the implied Riemann-Roch spaces and their associated representations are generic. Moreover, our set up enables us to interpolate with polynomials. This makes our algorithms closer to well-known algorithms based on polynomial interpolation such as Karatsuba or Cook.  

This paper begins with an overview of the current generalizations of CCMA. Section 3 focuses on the multiplication in small extensions. We explain how to reach the equality in the Winograd-De Groote bound with our construction. Moreover, this construction naturally integrates the trick of Karatsuba algorithm. In Section 4, we give a Recursive Polynomial Generic Construction (RPGC) of algorithms for the multiplication in any extension of $\mathbb{F}_q$, and give a natural strategy to build algorithms with a good bilinear complexity. In Section 5, we give a particular explicit construction of such algorithms having a quasi-linear uniform bound for their bilinear complexities, with respect to the extension degrees. Then, we show that the construction of these algorithms is deterministic, and give a polynomial asymptotic bound for this construction. 

\section{Chudnovsky and Chudnovsky Multiplication Algorithm}

A large description of CCMA and its generalizations is given in \cite{survey}. We first recall some basics of function field theory and introduce the notions required for our study. Then, we recall a specialized version of the generalized theorem/algorithm over a function field of arbitrary genus $g$, which will be useful for the proposed construction.

Let $F/\mathbb{F}_q$ be a function field of genus $g$ over $\mathbb{F}_q$. For $\mathcal{O}$ a valuation ring, the place $P$ is defined to be $P=\mathcal{O} \setminus \mathcal{O}^\times$. We denote by $F_P$ the residue class field at the place $P$, that is isomorphic to $\mathbb{F}_{q^d}$, $d$ being the degree of the place. A rational place is a place of degree $1$. We also denote by $B_d(F/{\mathbb F}_q)$ the number of places of degree $d$ of $F$ over ${\mathbb F}_q$. A divisor $\D$ is a formal sum $\D=\sum_i n_{i} P_i$, where $P_i$ are places and $n_i$ are relative integers. The support $supp~\D$ of $\D$ is the set of the places $P_j$ for which $n_{j}\neq 0$, and $\D$ is effective if all the $n_i$ are positive. The degree of $\D$ is defined by $\deg \D=\sum_i n_{i}$. The Riemann-Roch space associated to the divisor $\D$ is denoted by  ${\mathcal L}(\D)$. A divisor $\D$ is said to be non-special if $\dim\mathcal{L}(\D)=\deg(\D)+1-g$. Details about algebraic function fields can be found in \cite{stic2}. 

Since Ballet and Rolland \cite{baro1}, Arnaud \cite{arna1}, then Cenk and Özbudak \cite{ceoz} and finally the best current generalization due to Randriambololona \cite{randJComp}, the algorithm has been extended to the evaluation at places of arbitrary degrees and with multiplicity greater than 1. We recall that the generalized evaluation map is defined by the following:

\begin{definition}\label{genmap}
For any divisor $\D$, $P$ a place of degree $d$ and the multiplicity $u\geq1$ an integer, we define the generalized evaluation map 
\begin{equation}
\label{gen_eval}
\varphi_{\D,P,u}:\left|\begin{array}{ccl}
\Ld{} & \longrightarrow & (\F_{q^d})^u\\
f & \mapsto & (f(P),f'(P),\dots,f^{(u-1)}(P))
\end{array}\right.
\end{equation}
where the $f^{(k)}(P)$ are the coefficients of the local expansion
\begin{equation}\label{localexp}
f=f(P)+f'(P)t_P+f''(P)t_P^2+\cdots+f^{(k)}(P)t_P^k+\cdots
\end{equation}
of $f$ at $P$ with respect to the local parameter $t_P$, i.e. in $\mathbb{F}_{q^d}[[t_P]]$. 
\end{definition}

This map is also called a ``derivative evaluation map at order $u$''. In particular, the notation $f(P)$ denotes the residue of $f$ in $F_P$, that we often call the evaluation at P. Now, we define the generalized Hadamard product.

\begin{definition}
Let $q$ be a prime power and $d_1,\ldots,d_N,u_1,\ldots,u_N$ be positive integers. The generalized Hadamard product in $(\mathbb F_{q^{d_1}})^{u_1} \times \dots \times (\mathbb F_{q^{d_N}})^{u_N}$, denoted by $\underline{\odot}$, is given for all $(a_1,\ldots,a_N),(b_1,\ldots,b_N)\in(\mathbb F_{q^{d_1}})^{u_1} \times \dots \times (\mathbb F_{q^{d_N}})^{u_N}$ by
$$(a_1,\ldots,a_N)\underline{\odot}(b_1,\ldots,b_N)=(a_1b_1,\ldots,a_Nb_N).$$

\end{definition}

In the following, each product $a_ib_i$ in $(\mathbb{F}_{q^{d_i}})^{u_i}$ is the truncated product of two local expansions at a place $P$: the $u_i$ first elements of the product of two elements of the form of (\ref{localexp}) in $\mathbb{F}_{q^{d_i}}[[t_P]]$, i.e. the product in ${\mathbb{F}_{q^d}[[t_P]]}/{(t_P^{u_i})}$. Following the notation introduced in \cite{randJComp}, we denote by $\mu_q(d_i,u_i)$ the bilinear complexity of such truncated product. 
Now, let us introduce a specialized version of the current generalization of CCMA.

\begin{theorem}[CCMA at places of arbitrary degrees with derivative evaluations]\label{Algogene}

Let \begin{itemize}
      \item $n$ be a positive integer,
      \item $F /\F_{q}$ be an algebraic function field of genus $g$,
      \item $Q$ be a degree $n$ place of $F /\F_{q}$,
      \item $\D$ be a divisor of $F /\F_{q}$,
      \item ${\mathcal P}=\{P_1,\ldots , P_N\}$ be an ordered set of places of arbitrary degrees of $F /\F_{q}$,
      \item $\underline{u}=(u_1,\ldots,u_N)$ be positive integers.
   \end{itemize}

We suppose that $supp\;\D \cap \{Q,P_1,...,P_N\}=\emptyset$ and that
\begin{enumerate}
\item[(i)] the evaluation map
$$\begin{array}{lccc}
 Ev_Q: & \mathcal L(\D) & \rightarrow &   F_Q \\
    &  f              & \mapsto     & f(Q)
\end{array}$$
is surjective,
\item[(ii)] the evaluation map
$$\begin{array}{lccc}
 Ev_{\mathcal P}: & \mathcal L(2\D) & \rightarrow & (\mathbb F_{q^{\deg P_1}})^{u_1} \times \dots \times (\mathbb F_{q^{\deg P_N}})^{u_N}  \\
    &  f              & \mapsto     & \left(\strut \varphi_{2\D,P_1,u_1}\left(f\right),\ldots,\varphi_{2\D,P_N,u_N}\left(f\right)\right)
\end{array}$$
is injective.
\end{enumerate}
Then,
\begin{itemize}
\item [(1)]we have a multiplication algorithm $\mathcal{U}_{q,n}^{F,\mathcal{P},\underline{u}}(\D,Q)$ such that for any two elements $x$, $y$ in $\mathbb F_{q^n}$:

\begin{equation}\label{directproductalgoChud}
xy= E_Q \circ Ev_{\mathcal P}{|_{Im Ev_{\mathcal P}}}^{-1} \left( E_{\mathcal P}\circ Ev_Q^{-1}(x) \underline{\odot} E_{\mathcal P}\circ Ev_Q^{-1}(y)  \right),
\end{equation}
where $E_Q$ denotes the canonical projection from the valuation ring ${\mathcal O}_Q$ of the place $Q$ in its residue class field $F_Q$, $E_{\mathcal P}$ the extension of $Ev_{\mathcal P}$ on the valuation ring ${\mathcal O}_Q$ of the place $Q$, $Ev_{\mathcal P}{|_{Im Ev_{\mathcal P}}}^{-1}$ the restriction of the inverse map of $Ev_{\mathcal P}$ on its image, $\underline{\odot}$ the generalized Hadamard product and $\circ$ the standard composition map; 
%}}
\item [(2)] the algorithm $\mathcal{U}_{q,n}^{F,\mathcal{P},\underline{u}}(\D,Q)$ defined by \eqref{directproductalgoChud} has bilinear complexity $$\mu(\mathcal{U}_{q,n}^{F,\mathcal{P},\underline{u}}(\D,Q))= \sum_{i=1}^N \mu_q(\deg P_i,u_i).$$
\end{itemize}
\end{theorem}

\noindent Moreover, recall that sufficient application conditions are given in \cite{survey}:

\begin{theorem}\label{criteres}
Existence of the objects satisfying the conditions of Theorem \ref{Algogene} above is ensured by the following numerical criteria:
\begin{itemize}
    \item [(a)] a sufficient condition for the existence of a place $Q$ in $F/\mathbb{F}_q$ of degree $n$ is that $2g+1\leq q^{(n-1)/2}(q^{1/2}-1)$, where $g$ is the genus of $F$,
    \item [(b)] a sufficient condition for $(i)$  is that the divisor $D-Q$ is non-special,
    \item [(c)] a necessary and sufficient condition for $(ii)$ is that the divisor $2\D-\mathcal G$ is zero-dimensional:
    $$\dim \mathcal L(2\D-\mathcal G) = 0$$
    where $\mathcal{G}=u_1P_1+\cdots +u_NP_N$.
\end{itemize}
\end{theorem}

These results include the algorithm without derivative evaluations, setting $u_i=1$. The following corollary provides some sufficient conditions particularly useful for what follows.

\begin{corollary}[Criteria for CCMA at places of arbitrary degrees without derivative evaluation]\label{Algodegresupnoderiv}

Let $q$ be a prime power and let $n$ be an integer $>1$. 
If there exists an algebraic function field $F/\F_q$ of genus $g$ with a set of places $\mathcal{P}=\{P_1,\ldots,P_N\}$ and an effective divisor $\D$ of degree $n+g-1$ such that

\begin{enumerate}
        \item[1)] there exists a place $Q$ of degree $n$
         (which is always the case if $2g+1 \leq q^{\frac{n-1}{2}}(q^{\frac{1}{2}}-1)$), 
        \item[2)] $Supp~\D\cap(\mathcal{P}\cup Q)=\emptyset$, and $D-Q$ is non-special,
        \item[3)] $\sum_{i=1}^N\deg P_i = 2n+g-1$ and $2\D-\sum P_i$ is non-special,
\end{enumerate}

then,

\begin{enumerate}
\item[(i)] the evaluation map
$$\begin{array}{lccc}
 Ev_Q: & \mathcal L(\D) & \rightarrow & \frac {\mathcal O_Q}{Q}  \\
    &  f              & \mapsto     & f(Q)
\end{array}$$
is an isomorphism of vector spaces over $\F_q$,
\item[(ii)] and the evaluation map
$$\begin{array}{lccl}
 Ev_{\mathcal P}: & \mathcal L(2\D) & \rightarrow & \mathbb \mathbb F_{q^{\deg P_1}} \times \dots \times \mathbb F_{q^{\deg P_N}} \\
    &  f              & \mapsto     & \left(\strut f\left(P_1\right),\ldots,f\left(P_{N}\right)\right)
\end{array}$$
is an isomorphism of vector spaces of dimension $2n+g-1$ over $\F_q$.
\end{enumerate}
\end{corollary}
Conditions $1)$, $2)$ and $3)$ of Corollary \ref{Algodegresupnoderiv} gives the conditions $a)$, $b)$ and $c)$ of Theorem \ref{criteres} respectively. Note that all these new requirements are not necessary to construct the algorithm. Nevertheless, it corresponds to interesting conditions under which we want to build our algorithms.
In particular, if we use only interpolation on rational places, we obtain the criteria for the original CCMA \cite{ball1}.
\begin{corollary}[Criteria for the original CCMA]
\label{originalchud}
Let $q$ be a prime power and let $n$ be an integer $>1$. 
If there exists an algebraic function field $F/\F_q$ of genus $g$ satisfying the conditions 

\begin{enumerate}%[1)]
         \item  $B_n(F/\F_q)>0$
         (which is always the case if $2g+1 \leq q^{\frac{n-1}{2}}(q^{\frac{1}{2}}-1)$),
	\item $B_1(F/\F_q) > 2n+2g-2$, 
\end{enumerate}

then there exists a divisor $D$ of degree $n+g-1$, a place $Q$ of degree $n$ and a set of rational places $\mathcal{P}$ such that $(i)$ and $(ii)$ of Corollary \ref{Algodegresupnoderiv} holds. 
\end{corollary}

\noindent In the following, we specialize these results to the rational function field $\mathbb{F}_q(x)$.

\section{CCMA and the multiplication in small extensions of $\mathbb{F}_q$}\label{smallext}

\subsection{Polynomial interpolation over rational points}\label{polyinter}

As seen in the introduction, the multiplication in any extension of $\mathbb{F}_q$ of degree $n\leq\frac{1}{2}q+1$ requires exactly $2n-1$ bilinear multiplications \cite{groo}, and every algorithm reaching this optimal bilinear complexity is of type interpolation \cite{wino3}. 
In this section, we construct Chudnovsky-type algorithms over the projective line using polynomial interpolation, and which have optimal bilinear complexity for $q$ a prime power and $n\leq\frac{1}{2}q+1$. We begin with the following set up. 

\medskip
\begin{center}
 \rule{\linewidth}{1pt}
 \end{center}
\textbf{PGC : Polynomial Generic Construction}

\noindent For $q$ a prime power and $n<\frac{1}{2}q+1$ a positive integer. We set
\begin{itemize}
    \item $Q$ is a place of degree $n$ of $\mathbb{F}_q(x)$,
    \item $\D=(n-1)P_\infty$,
    \item $\mathcal{P}$ is a set of rational places distinct from $P_\infty$ of cardinal $|\mathcal{P}|=2n-1$,
    \item the basis of $\mathcal{L}(\D)$ is $\{1,x,\ldots,x^{n-1}\}$, and 
    \item the basis of $\mathcal{L}(2\D)$ is $\{1,x,\ldots,x^{2n-1}\}$.
\end{itemize}
\begin{center}
 \rule{\linewidth}{1pt}
 \end{center}
 
In our construction, we set the function field to be $\mathbb{F}_q(x)$, and the divisor to be $\D=(n-1)P_\infty$. In order to define an algorithm for the multiplication in $\mathbb{F}_{q^n}$ with Theorem \ref{Algogene}, the only variables left  are the place $Q$, the set $\mathcal{P}$ and the integers in $\underline{u}$. Hence, we denote the algorithm using these parameters by $\mathcal{U}_{q,n}^{\mathcal{P},\underline{u}}(Q)$. When we do not evaluate with multiplicity, i.e. all $u_i$ are equal to 1, we denote the algorithm by $\mathcal{U}_{q,n}^{\mathcal{P}}(Q)$ to lighten the notations.

\begin{proposition}
Let $q$ be a prime power, $n<\frac{1}{2}q+1$ be an integer and $\mathcal{P}$ be a set of rational places distinct from $P_\infty$ of cardinal $|\mathcal{P}|=2n-1$. Then, PGC is a set-up for a CCMA from Corollary \ref{originalchud}, denoted by $\mathcal{U}_{q,n}^\mathcal{P}(Q)$, for the multiplication in $\mathbb{F}_{q^n}$. This algorithm  interpolates over polynomials and computes $2n-1$ bilinear multiplications in $\mathbb{F}_q$.
\end{proposition}

\begin{proof}
First, finding a degree $n$ place $Q$ to construct $\mathbb{F}_{q^n}$ means finding a monic irreducible polynomial $Q(x)$ of degree $n$ over $\mathbb{F}_q$. The residue class field at $Q$ is exactly the quotient $\mathbb{F}_q[x]/(Q(x))=\mathbb{F}_{q^n}$. Such polynomials exist for all $q$ and $n$, and the condition $1.$ of Corollary \ref{originalchud} is verified. 

Let $P_\infty$ be the place at infinity of $\mathbb{F}_q(x)$, and $\D$ be the divisor defining the Riemann-Roch space. We set $\D=(n-1)P_\infty$. Then, $\{1,x,\ldots,x^{n-1}\}$ is canonically a basis of $\mathcal{L}(\D)$, and $\{1,x,\ldots,x^{2n-2}\}$ is a basis of $\mathcal{L}(2\D)$, and the interpolation is done with polynomials. 

There are $q+1$ rational places in $\mathbb{F}_q(x)$. The condition $2.$ of Corollary \ref{originalchud} attests that the algorithm can be built for $n\leq\frac{1}{2}q+1.$
But in PGC, the place at infinity $P_\infty$ is used to define the divisor $\D$. It implies that we can evaluate only on $q$ rational places instead of $q+1$, since a function in $\mathcal{L}(\D)$ has a pole at $P_\infty$. Thus, we use the more general Corollary \ref{Algodegresupnoderiv}. The divisor $\D-Q$ is of degree $-1$ and hence is non-special (\cite{stic2}, Remark 1.6.11), thus condition $2)$ is verified. Moreover, let $\mathcal{P}=\{P_i\}$, be a set of rational places distinct from $P_\infty$. Condition $3)$ requires that the cardinal $\mid\mathcal{P}\mid =2n-1\leq q$, so $2\D-\sum P_i$ is of degree $-1$ and thus non-special. Hence, Corollary \ref{Algodegresupnoderiv} attests that the algorithm can be built with PGC for $n<\frac{1}{2}q+1.$

Therefore, for all prime powers $q$, PGC gives CCMA with optimal bilinear complexity using polynomial interpolation for the multiplication in an extension of $\mathbb{F}_q$ of degree $n<\frac{1}{2}q+1$.\end{proof}

\begin{remark}\label{boarderline}
CCMA cannot be constructed with PGC when $n=\frac{1}{2}q+1$.
\end{remark}

When $q$ is odd, this equality never happens, because $\frac{1}{2}q+1$ is not an integer. 
When $q$ is even, the same construction is not possible in the borderline case $n=\frac{1}{2}q+1$. For an even $q\geq4$, we can use a place $\mathcal R$ of degree $n-1$ to define the divisor, i.e. set $\D=\mathcal{R}$. With this setting, we can evaluate at $P_\infty$ and hence on the $q+1=2n-1$ rational places of $\mathbb{F}_q(x)$. Then, an algorithm of multiplication is constructed with optimal bilinear complexity. Nevertheless, the basis of the Riemann-Roch space $\mathcal{L}(2\mathcal{R})$  will be some linear combinations of $\mathcal{B}=\{\frac{1}{\mathcal R^2(x)},\frac{x}{\mathcal{R}^2(x)},\ldots,\frac{x^{2n-1}}{\mathcal R^2(x)}\}$, where $\mathcal{R}(x)$ is the monic irreducible polynomial of degree $n-1$ defining $\mathcal{R}$. In fact, if we denote by $v_\infty$ the valuation at $P_\infty$, we obtain $v_\infty(\frac{x^i}{\mathcal R^2(x)})=i-2(n-1)$. Then, two distinct elements of $\mathcal{B}$ have two different valuations at $P_\infty$, and $\mathcal{B}$ is a basis of $\mathcal{L}(2\mathcal{R})$. Hence, we obtain an algorithm of bilinear complexity $2n-1$, but that interpolates no longer with polynomials but with rational functions (this construction is illustrated in the forthcoming Example \ref{abusex}). In the next section, we see how to obtain a polynomial interpolation algorithm in this case.

\subsection{The case of $n=\frac{1}{2}q+1$ and polynomial interpolation}\label{alaK}

We consider the case of Remark \ref{boarderline}: the extension of $\mathbb F_q$ of degree $n=\frac{1}{2}q+1$. We want to build a Chudnovsky-type algorithm over the rational function field $\mathbb{F}_q(x) $, demanding $\D=(n-1)P_\infty$ to interpolate with polynomials. From the results of Winograd and De Groote, it must be possible to construct such an algorithm with optimal bilinear complexity. Hence, we want to find a way to get back the evaluation at $P_\infty$, which is not allowed in our construction. We use the fact that the leading coefficient of the product is the product of the leading coefficients. In fact, since $\D=(n-1)P_\infty$, functions in $\mathcal{L}(\D)$ are polynomials of degrees $n-1$, and the product of two of them belongs to $\mathcal{L}(2\D)$ and is of degree $2n-2$. Hence, we can canonically use $\{1,x,\ldots,x^{n-1}\}$ as a basis for $\mathcal L(\D)$, and $\{1,x,\ldots,x^{2n-2}\}$ as a basis of $\mathcal{L}(2\D)$. Let $f=\sum_{i=0}^{n-1}a_ix^i$ and $g=\sum_{i=0}^{n-1}b_ix^i$ be functions in $\mathcal{L}(\D)$, and their product be $fg=\sum_{i=0}^{2n-2}c_ix ^i\in\mathcal{L}(2\D)$. Its leading coefficient $c_{2n-2}$ is equal to $a_{n-1}b_{n-1}$. We define $P_0$ to be the place associated to the polynomial $x$. Then, $x$ is a local parameter for the expansion at the place $P_0$, and a function in $\mathcal L(\D)$ in the previous basis is its own Laurent expansion at $P_0$. Hence, we can interpret the product of the leading coefficients in terms of derivative evaluations at $P_0$:
\begin{equation}\label{trickgene} f^{(n-1)}(P_0)g^{(n-1)}(P_0)=(fg)^{(2n-2)}(P_0),\end{equation}
where $f^{(i)}$ is the $i-th$ coefficient of the Laurent expansion, as in Definition \ref{genmap}.
We use this trick to overcome the incapacity to evaluate at the place at infinity. Let us introduce the following notation.

\begin{definition}
\label{abus}
Let $k$ be a positive integer and $P_\infty$ be the place at infinity of $\mathbb{F}_q(x)$. Set $\mathcal{L}(\D)=\mathcal{L}(kP_\infty)$, we define the evaluation at $P_\infty$ to be for all $f\in\mathcal{L}(\D)$, $$f_{\D}(P_\infty):=f^{(k)}(P_0),$$ the $k+1-$th coefficient of the Laurent expansion at $P_0$, that is also the leading coefficient of $f$. We specify the divisor $\D$ in the notation as the evaluation depends on the Riemann-Roch space from which it is defined.
\end{definition}

\noindent Under these notations, the formula (\ref{trickgene}) becomes 
\begin{equation}\label{TRICK}
    f_{\D}(P_\infty)g_{\D}(P_\infty)=(fg)_{2\D}(P_\infty).
\end{equation}
In order to illustrate the legitimacy of this definition, we give an example of construction of an evaluation map $\Tilde{Ev_\mathcal{P}}$ using this trick, when the evaluations are done without multiplicity. 

\begin{example}\label{abusex}
Let $q\geq2$ be a prime power, $n\geq2$ be an integer, and $\mathbb{F}_q(x)$ be the rational function field. Let $\mathcal{P}=\{P_\infty,P_0,P_1,\ldots,P_N\}$ be a set of places of $\mathbb{F}_q(x)$, such that $\sum_{P\in\mathcal P}\deg P=2n-1$, where $P_0$ is the rational place associated to the polynomial $x$, and $P_\infty$ is the place at infinity. Set $\D=(n-1)P_\infty$. We consider the application
$$\begin{array}{lccc}
\Tilde{Ev_\mathcal P} & :\mathcal{L}(2\D) & \longrightarrow & \mathbb F_{q} \times \mathbb{F}_q \times \mathbb{F}_{q^{d_1}} \times \dots \times \mathbb F_{q^{d_N}} \\
& f &\mapsto&\left ( f_{2\D}(P_\infty),f(P_0),f(P_1),\ldots,f(P_{N})\right),\end{array}$$ 
where $f_{2\D}(P_\infty)=f^{(2n-2)}(P_0)$ is the leading coefficient of $f$. Suppose that $\mathcal{P}$ does not contain all places of degree $n-1$. Let $\mathcal{R}$ be a degree $n-1$ place of $\mathbb{F}_q(x)$ not included in $\mathcal{P}$, and $\mathcal{R}(x)$ is the corresponding monic irreducible polynomial of degree $n-1$ over $\mathbb{F}_q$. 
The Riemann-Roch spaces $\mathcal{L}(2\D)$ and $\mathcal{L}(2\mathcal{R})$ are of same dimension over $\mathbb{F}_q$ and isomorphic as vector spaces. The set $\mathcal{B}_{2\D}=\{x^i\}_{i=0,\ldots,2n-2}$ is a basis of $\mathcal{L}(2(n-1)P_\infty)$ and $\mathcal{B}_{2\mathcal{R}}=\{\frac{x^i}{\mathcal{R}(x)^2}\}_{i=0,\ldots,2n-2}$ is a basis of $\mathcal{L}(2\mathcal{R})$. The natural isomorphism between these two vector spaces is given by $\phi_{2\mathcal{R}}: x^i\mapsto x^i/\mathcal{R}(x)^2$, for all $x^i\in\mathcal{B}_{2\D}$.
Recalling that $\sum_{P\in\mathcal{P}}\deg P=2n-1$, the application 
\begin{equation}
\begin{array}{lccc}
 Ev_{2\mathcal{R}}: & \mathcal L(2\mathcal R) & \rightarrow & \mathbb F_{q} \times \mathbb{F}_q \times \mathbb{F}_{q^{d_1}} \times \dots \times \mathbb F_{q^{d_N}}  \\
    &  f              & \mapsto     & \left(\strut f(P_\infty),f(P_0),f\left(P_1\right),\ldots,f\left(P_{N}\right)\right),
\end{array}
\end{equation}
is injective because the divisor $(2\mathcal{R}-\sum_{P\in\mathcal{P}} P)$ is of negative degree, and bijective since the two vector spaces are of dimension $2n-1$. Note that this application corresponds to $Ev_{\mathcal{P}}$ in Theorem \ref{Algogene}, but we denote it here by $Ev_{2\mathcal R}$ to highlight its source. Now, consider
$$
\begin{array}{lccc}
  Ev_{2\mathcal{R}}\circ \phi_{2\mathcal{R}}: & \mathcal L(2\D) & \rightarrow & \mathbb F_{q} \times \mathbb{F}_q \times \mathbb{F}_{q^{d_1}} \times \dots \times \mathbb F_{q^{d_N}}\\
    &  f              & \mapsto     & \left(\strut \frac{f}{\mathcal R^2}(P_\infty),\frac{f}{\mathcal R^2}(P_0),\frac{f}{\mathcal R^2}(P_1),\ldots,\frac{f}{\mathcal R^2}(P_N)\right),
\end{array}
$$
where $\frac{f}{\mathcal{R}^2}$ denotes the rational function $\frac{f(x)}{\mathcal{R}(x)^2}\in \mathcal{L}(2\mathcal{R})$.
Let $f=\sum_{i=0}^{2n-2}a_ix^i$ be a function in $\mathcal{L}(2\D)$. We have
$\frac{f}{\mathcal{R}^2}(P_\infty)=a_{2n-2}=f_{2\D}(P_\infty)$, for all $f\in\mathcal{L}(2\D)$. Hence, the evaluation $f_{2\D}(P_\infty)$ in $\Tilde{Ev_P}$ corresponds exactly to $\frac{f}{\mathcal{R}^2}(P_\infty)$ in $Ev_{2\mathcal{R}}$. This justifies our motivation to write $f_{2\D}(P_\infty):=f^{(2n-2)}(P_0)$ the leading coefficient of $f$.
Moreover, the function $\frac{1}{\mathcal{R}^2(x)}$ belongs to $\mathcal{L}(2\mathcal{R})$. For all places $P_i\neq P_\infty$ in $\mathcal P$, $\frac{1}{\mathcal{R}^2}(P_i)$ does not vanish and we have $f(P_i)=\frac{f}{\mathcal{R}^2}(P_i)\frac{1}{\mathcal{R}^2}(P_i)^{-1}$. Finally,  $\Tilde{Ev_\mathcal{P}}(f)=\varphi_{2\mathcal{R}}\circ Ev_{2\mathcal{R}}\circ\phi_{2\mathcal{R}}(f)$, with 
$$\begin{array}{lccc}
 \varphi_{2\mathcal{R}}:    &  \mathbb F_{q} \times \mathbb{F}_q \times \mathbb{F}_{q^{d_1}} \times \dots \times \mathbb F_{q^{d_N}} & \hspace{-11pt} \rightarrow & \mathbb F_{q} \times \mathbb{F}_q \times \mathbb{F}_{q^{d_1}} \times \dots \times \mathbb F_{q^{d_N}} \\
     & (a_\infty,a_0,a_1,\ldots,a_N) &  \hspace{-11pt} \mapsto & \hspace{-10pt} (a_\infty,\mathcal{R}^2(P_0)a_0,\mathcal{R}^2(P_1)a_1,\ldots,\mathcal{R}^2(P_N)a_N).
\end{array}$$
\begin{center}
\begin{tikzpicture}
  \tikzstyle{stateEdgePortion} = [black,thick];
  \tikzstyle{stateEdge} = [stateEdgePortion,->];
  \tikzstyle{edgeLabel} = [pos=0.5, text centered, font={\sffamily\small}];
    \node[name=LD] {$\mathcal{L}(2\D)$};
    \node[name=LR, right of=LD, node distance=10em] {$\mathcal L(2\mathcal R)$};     \node[name=FqN, below of=LR, node distance=10em] {$\mathbb{F}_q^{2n-1}$};
    \node[name=Fqn, below of=LD, node distance=10em] {$\mathbb{F}_q^{2n-1}$};
    
  \draw ($(LD.east) + (0em,0em)$) 
      edge[stateEdge] node[edgeLabel,yshift=1em]{$\phi_{2\mathcal{R}}$} 
      ($(LR.west) +  (0em,0em)$);
  \draw ($(LD.south) + (0em,0em)$) 
      edge[stateEdge] node[edgeLabel,xshift=-2em]{$\Tilde{Ev_\mathcal{P}}$} 
      ($(Fqn.north) +  (0em,0em)$);
  \draw ($(LR.south) + (0em,0em)$) 
      edge[stateEdge] node[edgeLabel,xshift=2em]{$Ev_{2\mathcal{R}}$} 
      ($(FqN.north) +  (0em,0em)$);
    \draw ($(FqN.west) + (0em,0em)$) 
      edge[stateEdge] node[edgeLabel,yshift=-1em]{$\varphi_{2\mathcal{R}}$} 
      ($(Fqn.east) +  (0em,0em)$);
   
\end{tikzpicture}
\end{center}
Hence the previous diagram is commutative, and $\Tilde{Ev_P}$ is bijective from $\mathcal{L}(2\D)$ to $\mathbb{F}_q^{2n-1}$. Before constructing the algorithm using this evaluation map, note that we can use the divisor $\mathcal{R}$ to obtain an algorithm with Theorem \ref{Algogene}. In fact, let $Q$ be a place of degree $n$. Then, the place $Q$, the divisor $\mathcal{R}$ and the set $\mathcal{P}$ verify the conditions of Corollary \ref{Algodegresupnoderiv} of Theorem \ref{Algogene}. Thus, an algorithm of multiplication in $\mathbb{F}_{q^n}$ is given by
\begin{equation}\label{algoabusex}
    xy= E_Q \circ Ev_{2\mathcal{R}}^{-1} \left( E_{\mathcal{R}}\circ Ev_Q^{-1}(x) \underline{\odot} E_{\mathcal{R}}\circ Ev_Q^{-1}(y)  \right),
\end{equation}
where $E_{\mathcal{R}}$ is the restriction of $E_{2\mathcal{R}}$ to $\mathcal{L}(\mathcal{R})$, and the rest is defined as in Theorem $\ref{Algogene}$, with $Ev_Q$ from $\mathcal{L}(2\mathcal{R})$ to $\mathbb{F}_{q^n}$.
The interpolation is done with rational functions.
\end{example}

\begin{remark}
The above example shows how to substitute the evaluation at $P_\infty$ of a function in $\mathcal{L}(kP_\infty)$, for any positive integer $k$. In this case, the algorithm of Theorem \ref{Algogene} is modified as follows. 
\end{remark}

\begin{proposition}[Polynomial interpolation CCMA on the projective line]\label{algoavecPinfty}
Let \begin{itemize}
      \item $\mathbb{F}_q(x)$ be the rational function field over $\mathbb{F}_q$,
      \item $n$ be a positive integer,
      \item $Q$ be a degree $n$ place of $\F_q(x)$,
      \item ${\mathcal P}=\{P_\infty,P_0,P_1,\ldots , P_N\}$ be an ordered set of places of arbitrary degrees of $\F_{q}(x)$, with $P_\infty$ the place at infinity and $P_0$ associated to the polynomial $x$,
      \item $u_0,u_1,\ldots,u_N$ be positive integers, with $u_0<n-1$.
   \end{itemize}

We set $\D=(n-1)P_\infty$. If \begin{equation}\label{condalg}
    \sum_{i=0}^{N}u_i\deg P_i=2n-2,
\end{equation}
then
\begin{enumerate}
\item[(i)] the evaluation map
$$\begin{array}{lccc}
 Ev_Q: & \mathcal L(\D) & \rightarrow &   F_Q \\
    &  f              & \mapsto     & f(Q)
\end{array}$$
is bijective,
\item[(ii)] the evaluation map
$$\begin{array}{lccc}
 \Tilde{Ev_{\mathcal P}}: & \mathcal L(2\D) & \rightarrow & \mathbb F_{q} \times \mathbb F_{q}^{u_0} \times (\mathbb F_{q^{\deg P_1}})^{u_1} \times \dots \times (\mathbb F_{q^{\deg P_N}})^{u_N}  \\
    &  f              & \mapsto     & \left(\strut f_{2\D}(P_\infty),\varphi_{2\D,P_0,u_0}(f), \varphi_{2\D,P_1,u_1}\left(f\right),\ldots,\varphi_{2\D,P_N,u_N}\left(f\right)\right)
\end{array}$$
is an isomorphism of vector spaces.
\end{enumerate}
Moreover,
\begin{itemize}
\item [(1)]for any two elements $x$, $y$ in $\mathbb F_{q^n}$, we have a multiplication algorithm $\mathcal{U}_{q,n}^{\mathcal{P},\underline{u}}(Q)$ of type polynomial interpolation such that:
\begin{equation}\label{directproductmodif}
xy= E_Q \circ \Tilde{Ev_{\mathcal P}}^{-1} \left(\Tilde{ E_{\mathcal P}}\circ Ev_Q^{-1}(x) \underline{\odot} \Tilde{E_{\mathcal P}}\circ Ev_Q^{-1}(y)  \right),
\end{equation}
where $E_Q$ denotes the canonical projection from the valuation ring ${\mathcal O}_Q$ of the place $Q$ in its residue class field $F_Q$, $\Tilde{Ev_{\mathcal P}}^{-1}$ the inverse map of $\Tilde{Ev_{\mathcal P}}$, $\underline{\odot}$ the generalized Hadamard product in $\mathbb F_{q} \times \mathbb F_{q}^{u_0} \times (\mathbb F_{q^{\deg P_1}})^{u_1} \times \dots \times (\mathbb F_{q^{\deg P_N}})^{u_N}$, $\circ$ the standard composition map, and 
$$\begin{array}{lccc}
 \Tilde{E_{\mathcal P}}: & \mathcal L(\D) & \rightarrow & \mathbb F_{q} \times \mathbb F_{q}^{u_0} \times (\mathbb F_{q^{\deg P_1}})^{u_1} \times \dots \times (\mathbb F_{q^{\deg P_N}})^{u_N}  \\
    &  f              & \mapsto     & \left(\strut f_{\D}(P_\infty),\varphi_{\D,P_0,u_0}(f), \varphi_{\D,P_1,u_1}\left(f\right),\ldots,\varphi_{\D,P_N,u_N}\left(f\right)\right),
\end{array}$$
\item [(2)] the algorithm $\mathcal{U}_{q,n}^{\mathcal{P},\underline{u}}(Q)$ defined by \eqref{directproductmodif} has bilinear complexity $$\mu(\mathcal{U}_{q,n}^{\mathcal{P},\underline{u}})= \sum_{i=0}^N \mu_q(\deg P_i,u_i)+1.$$
\end{itemize}
\end{proposition}

\begin{proof}
Since $\D=(n-1)P_\infty$, $\mathcal{L}(\D)$ is isomorphic to $\mathbb{F}_{q^n}$, and we associate elements of $\mathcal{L}(\D)$ to elements of $\mathbb{F}_{q^n}$. For $f,g\in\mathcal L(\D)$, if we compute the Generalized Hadamard product of $\Tilde{E_\mathcal{P}(f)}$ and $\Tilde{E_\mathcal{P}(g)}$, we get

\[ \begin{aligned}
\Tilde{E_\mathcal P}(f) \odot \Tilde{E_\mathcal P}(g) & = \begin{pmatrix} 
f_{\D}(P_\infty) \\
\varphi_{\D,P_0,u_0}(f) \\
\varphi_{\D,P_1,u_1}\left(f\right)\\
\vdots\\
\varphi_{\D,P_N,u_N}\left(f\right)
\end{pmatrix}
\underline{ \odot}
\begin{pmatrix} 
g_{\D}(P_\infty) \\
\varphi_{\D,P_0,u_0}(g) \\
\varphi_{\D,P_1,u_1}\left(g\right)\\
\vdots\\
\varphi_{\D,P_N,u_N}\left(g\right)\end{pmatrix} \\ 
& =  \begin{pmatrix} 
f_{\D}(P_\infty)g_{\D}(P_\infty) \\
\varphi_{\D,P_0,u_0}(f)\varphi_{\D,P_0,u_0}(g) \\
\varphi_{\D,P_1,u_1}\left(f\right)\varphi_{\D,P_1,u_1}\left(g\right)\\
\vdots\\
\varphi_{\D,P_N,u_N}\left(f\right)\varphi_{\D,P_N,u_N}\left(g\right)\end{pmatrix}\\
& = \begin{pmatrix} 
(fg)_{2\D}(P_\infty) \\
\varphi_{2\D,P_0,u_0}(fg)\\
\varphi_{2\D,P_1,u_1}\left(fg\right)\\
\vdots\\
\varphi_{2\D,P_N,u_N}\left(fg\right)\end{pmatrix}=\Tilde{Ev_P}(fg).
\end{aligned}\]
We have to prove that $\Tilde{Ev_P}$ is bijective. First, the condition $u_0<n-1$ attests that $f^{(2n-2)}(P_0)$ is not involved in $\varphi_{2\D,P_0,u_0}(f)$. Let $f=a_0+a_1x+\cdots+a_{2n-2}x^{2n-2}$ be a function in $\mathcal{L}(2\D)$ such that $f\in\ker\Tilde{Ev_\mathcal P}$. In particular, $f^{(2n-2)}(P_0)=0=a_{2n-2}$. Then, $f=\sum_{i=0}^{2n-3}a_ix^i\in\mathcal{L}((2n-3)P_\infty$) and $$\left(\varphi_{2\D,P_0,u_0}(f), \varphi_{2\D,P_1,u_1}\left(f\right),\ldots,\varphi_{2\D,P_N,u_N}\left(f\right)\right)=(0,0,\ldots,0).$$ Hence, $f\in\ker \Tilde{Ev_P}\subseteq \mathcal{L}((2n-3)P_\infty-\sum_{i=0}
^N u_iP_i)$. But, the divisor $(2n-3)P_\infty-\sum_{i=0}
^N u_iP_i$ is of degree -1 so $f=0$ and $\ker\Tilde{Ev_\mathcal P}=\{0\}$. Thus   $\Tilde{Ev_\mathcal P}$ is injective, and bijective since between two vector spaces of same dimension. 
For all $f,g\in\mathcal{L}(\D)$, we obtain
$$fg=\Tilde{Ev_P}^{-1}(\Tilde{E_\mathcal P}(f) \underline{\odot} \Tilde{E_\mathcal P}(g)),$$
and finally, with $f=Ev_Q^{-1}(x)$ and $g=Ev_Q^{-1}(y)$ for any $x,y\in\mathbb{F}_{q^n}$,
$$xy= E_Q \circ \Tilde{Ev_{\mathcal P}}^{-1} \left( \Tilde{E_{\mathcal P}}\circ Ev_Q^{-1}(x)\underline{\odot} \Tilde{E_{\mathcal P}}\circ Ev_Q^{-1}(y)  \right),$$
where $E_Q$ denotes the canonical projection from the valuation ring ${\mathcal O}_Q$ of the place $Q$ in its residue class field $F_Q$. 
Hence, we obtain an algorithm of polynomial interpolation, since functions in the Riemann-Roch spaces are polynomials. Its bilinear multiplications are first given by the $\varphi_{\D,P_i,u_i}\left(f\right)\varphi_{\D,P_i,u_i}\left(g\right)$, that require $\mu_q(\deg P_i,u_i)$ bilinear multiplications over $\mathbb{F}_q$ for each $i=0,1,\ldots,N$. Then, one more bilinear multiplication is used to compute $f_{\D}(P_\infty)g_{\D}(P_\infty)$. These multiplications are independent of the choice of the place $Q$ chosen to represent $\mathbb{F}_{q^n}$. Thus the bilinear complexity of the algorithm $\mathcal{U}_{q,n}^{\mathcal{P},\underline{u}}(Q)$ is the same for all possible $Q$, and is given by $\mu(\mathcal{U}_{q,n}^{\mathcal{P},\underline{u}})= \sum_{i=0}^N \mu_q(\deg P_i,u_i)+1. $
\end{proof}

\begin{remark}
In fact, let $Q'$ be another place of degree $n$. In our construction, $Ev_Q$ and $Ev_{Q'}$ are both trivial. Then, the only difference between $\mathcal{U}_{q,n}^{\mathcal{P},\underline{u}}(Q)$ and $\mathcal{U}_{q,n}^{\mathcal{P},\underline{u}}(Q')$ is that $E_Q\neq E_{Q'}$. This is why we consider $\mathcal{U}_{q,n}^{\mathcal{P},\underline{u}}$ as an algorithm that can be applied to any place of degree $n$ of $\mathbb{F}_q(x)$, with bilinear complexity $\mu(\mathcal{U}_{q,n}^{\mathcal{P},\underline{u}})$.
\end{remark}

\begin{corollary}\label{coroalgosansED}
Without derivative evaluations, i.e. $\underline{u}=(1,\ldots,1)$, the condition (\ref{condalg}) becomes \begin{equation}\label{condsansed}\sum_{P\in\mathcal{P}}\deg P=2n-1,\end{equation}
and Proposition \ref{algoavecPinfty} gives an algorithm $\mathcal{U}_{q,n}^\mathcal{P}$ for the multiplication in $\mathbb{F}_{q^n}$, with bilinear complexity $\mu(\mathcal{U}_{q,n}^\mathcal{P})=\sum_{P\in\mathcal{P}}\mu_q(\deg P)$.
\end{corollary}

\begin{example}\label{suiteex1}
With the settings of Example \ref{abusex}, Corollary \ref{coroalgosansED} gives an algorithm $\mathcal{U}_{q,n}^{\mathcal{P}}(Q)$ with $\D=(n-1)P_\infty$, $Q$ and $\mathcal{P}$. The multiplication is given by the formula (\ref{directproductmodif}). We saw $\Tilde{Ev_\mathcal{P}}=\varphi_{2\mathcal{R}}\circ Ev_{2\mathcal{R}}\circ\phi_{2\mathcal{R}}$. Moreover, we can similarly argue that $\Tilde{E_\mathcal{P}}$ is equal to $\varphi_\mathcal{R}\circ E_{\mathcal{R}}\circ\phi_\mathcal{R}$, where  $\phi_\mathcal{R}$ and $\varphi_\mathcal{R}$ are defined the same way (using $\mathcal{R}(x)$ instead of $\mathcal{R}^2(x))$. Finally, the algorithm $\mathcal{U}_{q,n}^{\mathcal{P}}(Q)$ is defined by

$
 xy   = E_Q \circ \Tilde{Ev_{\mathcal{P}}}^{-1} \left( \Tilde{E_{\mathcal{P}}}\circ Ev_Q^{-1}(x)\underline{\odot} \Tilde{E_{\mathcal{P}}}\circ Ev_Q^{-1}(y)  \right)
$
$$
 =  E_Q\circ \phi_{2\mathcal{R} }^{-1}\circ Ev_{2\mathcal R}^{-1}\circ\varphi_{2\mathcal{R}}^{-1}\left(\varphi_\mathcal{R}\circ E_{\mathcal{R}}\circ\phi_\mathcal{R}\circ Ev_Q^{-1}(x)\underline{\odot} \varphi_\mathcal{R}\circ E_{\mathcal{R}}\circ\phi_\mathcal{R}\circ Ev_Q^{-1}(y)  \right),
$$
where $Ev_Q$ is this time defined from $\mathcal{L}(2\mathcal{D})$ to $\mathbb{F}_{q^n}$.
\end{example}

Using this result for the construction with rational places, we can include the place at infinity in our set up. We obtain an algorithm of type polynomial interpolation of optimal bilinear complexity when $n=\frac{1}{2}q+1$.

\begin{corollary}\label{caslimite}
Let $q\geq2$ be an even prime power and $n=\frac{1}{2}q+1 $. Let $\mathcal{P}$ be the set of all rational places of $\mathbb F_q(x)$. Given $Q$ a place of degree $n$, Corollary \ref{coroalgosansED} gives a Chudnovsky-type algorithm over the projective line $\mathcal{U}_{q,n}^\mathcal{P}(Q)$ for the multiplication in $\mathbb{F}_{q^n}$. This algorithm interpolates over polynomials and computes $2n-1$ bilinear multiplications in $\mathbb{F}_q$.
\end{corollary}

\begin{remark}\label{U43}
For $q=4$ and $n=3,$ this construction provides an algorithm $\mathcal{U}_{4,3}^{\mathcal{P}}$ with bilinear complexity $\mu(\mathcal{U}_{4,3}^{\mathcal{P}})=5$. This is one less bilinear multiplication than the result from Cenk and {\"O}zbudak (\cite{ceoz}, \cite{survey} Table 2). 
\end{remark}

\subsection{A particular case: the quadratic extension of $\mathbb{F}_2$}\label{quad}

The case of $q=2$ and $n=2$ is problematic and interesting. CCMA cannot be constructed with PGC for the multiplication in $\mathbb {F}_{2^2}$ over $\mathbb {F}_2$. In fact, the rational function field $\mathbb F_2(x)$ has only three rational places: $P_0$, the place associated to the polynomial $x$, $P_1$, associated to $x-1$, and $P_\infty$, the place at infinity. The proposed construction requires $P_\infty$ to define the Riemann-Roch space and three other places to evaluate. Thus, we cannot use a place $\mathcal{R}$ of degree $n-1$ to define the Riemann-Roch space, as in Examples \ref{abusex} and \ref{suiteex1}. We can use derivative evaluations and build an algorithm with Theorem \ref{Algogene}. Using evaluations at $P_0$ with multiplicity 2 and at $P_1$ with multiplicity 1, this gives an algorithm computing 4 bilinear multiplications, which are exactly those of the schoolbook method. That is one more than with the Karatsuba Algorithm, that gives for all prime power $q$ the optimal bilinear complexity $\mu_q(2)=3$. Corollary \ref{caslimite} gives a Chudnovsky-type algorithm reaching this bilinear complexity.

\begin{corollary}\label{F22}
Let $Q$ be the degree 2 place of $\mathbb{F}_2(x)$ and $\mathcal P=\{P_0,P_1,P_\infty\}$, where $P_0$ and $P_1$ are the places associated to $x$ and $x-1$ respectively, and $P_\infty$ is the place at infinity. Then, $\mathcal{U}_{2,2}^\mathcal{P}(Q)$ is a Chudnovsky-type algorithm for the multiplication in the quadratic extension of $\mathbb{F}_{2}$ with bilinear complexity $\mu(\mathcal{U}_{2,2})=3$. Moreover, its bilinear multiplications are corresponding to those of the Karatsuba Algorithm.
\end{corollary}
\begin{proof}
Corollary \ref{caslimite} gives us the existence of a Chudnovsky-type algorithm for the multiplication in the quadratic extension of $\mathbb{F}_2$ with 3 bilinear multiplications.  We give some details of the construction to show the correspondence with the Karatsuba Algorithm. 

Setting $\D=P_\infty$, a function in $\mathcal{L} (\D)$ is a degree one polynomial over $\mathbb F_2$, so we use $\{1,x\}$ as basis of $\mathcal L(\D)$, and $\{1,x,x^2\}$ as basis of $\mathcal{L}(2\D)$. Hence, there is an isomorphism $Ev_Q:\mathcal L(\D) \longrightarrow \mathbb F_{2^2}$. As $x$ is a local parameter for the expansion at the place $P_0$, a function in $\mathcal L(\D)$ in the previous basis is its own Laurent expansion at $P_0$. This way, if $f=f_0+f_1x\in \mathcal L(\D)$, we have $f'(P_0)=f_1$, and if $h=h_0+h_1x+h_2x^2\in\mathcal L(2\D)$, we have $h''(P_0)=h_2$. Following Definition \ref{abus}, we write $f_{\D}(P_\infty)=f_1$ and $h_{2\D}(P_\infty)=h_2$ the evaluations of the leading coefficients of $f$ and $h$ respectively. We use the trick of the previous section: the leading coefficient of the product (in $\mathcal{L}(2\D)$) of two elements in $\mathcal{L}(\D)$ is the product of their leading coefficients. Indeed, let $f=f_0+f_1x$ and $g=g_0+g_1x$ be two elements of $\mathcal L(\D)$. By (\ref{TRICK}), we have
$(fg)_{2\D}(P_\infty)=f_{\D}(P_\infty)g_{\D}(P_\infty).$
Let $f,g$ be functions in $\mathcal{L}(\D)$. Following Proposition \ref{algoavecPinfty}, we define $E_\mathcal P(f)=\left(f_{\D}(P_\infty),f(P_0),f(P_1)\right)$ and for $h\in\mathcal{L}(2\D)$, $\Tilde{Ev_{\mathcal P}}(h):=\left (h_{2\D}(P_\infty),h(P_0),h(P_1)\right)$ 
In this case, $\underline{\odot}$ is the classical Hadamard product $\odot$, i.e. the term by term product in $\mathbb{F}_q$. The bilinear multiplications are hence given by

\[ \begin{aligned}
E_\mathcal P(f) \odot E_\mathcal P(g) & = \begin{pmatrix} 
f_{\D}(P_\infty) \\
f(P_0) \\
f(P_1)\end{pmatrix}
\odot
\begin{pmatrix} 
g_{\D}(P_\infty) \\
g(P_0) \\
g(P_1)\end{pmatrix} \\ 
& =  \begin{pmatrix} 
f_{\D}(P_\infty)g_{\D}(P_\infty) \\
f(P_0)g(P_0) \\
f(P_1)g(P_1)\end{pmatrix}
 = \begin{pmatrix} 
(fg)_{2\D}(P_\infty) \\
(fg)(P_0) \\
(fg)(P_1)\end{pmatrix}
=\Tilde{Ev_{\mathcal P}}(fg).
\end{aligned}\]

Hence, for $f,g\in\mathcal{L}(\D)$ such that $f=Ev_Q^{-1}(x)$ and $g=Ev_Q^{-1}(y)$ for any $x,y\in\mathbb F_{2^2}$, we have $\Tilde{Ev_P}^{-1}(E_\mathcal P(f) \odot E_\mathcal P(g))=fg$. In fact, the product in $\mathcal{L}(2\D)$ is given by
$$\begin{array}{ll}
fg &=
f(P_0)g(P_0)+\big(f(P_1)g(P_1)-f(P_0)g(P_0)-f_{\D}(P_\infty)g_{\D}(P_\infty) \big)x\\&\hfill +f_{\D}(P_\infty)g_{\D}(P_\infty) x^2
\\ &=f_0g_0+\big((f_0+f_1)(g_0+g_1)-f_0g_0-f_1g_1\big)x+f_1g_1x^2,
\end{array}$$
and this multiplication of $f$ and $g$ is exactly that of the Karatsuba Algorithm.
\end{proof}

\begin{remark}\label{quadchud}
Corollary \ref{F22} becomes generalized to any prime power $q$. Let $P_0$, $P_1$ and $P_\infty$ be the places of $\mathbb{F}_q(x)$ associated to $x$, $x-1$ and at infinity. Let $\mathcal{P}_2=\{P_0,P_1,P_\infty\}$, and $Q$ be a degree 2 place of $\mathbb{F}_q(x)$. Then, $\mathcal{U}^{\mathcal{P}_2}_{q,2}(Q)$ is a Chudnovsky-type algorithm for the multiplication in the quadratic extension of $\mathbb{F}_q$, such that $\mu(\mathcal{U}_{2,2}^{\mathcal{P}_2})=3$. The bilinear multiplications of this algorithm are again exactly those of Karatsuba Algorithm. 
\end{remark}

\section{Recursive Chudnovsky-type algorithm on $\mathbb{F}_q(x)$}

\subsection{Recursive Polynomial Generic Construction}

Thus far, we have built polynomial interpolation Chudnovsky-type algorithms on the projective line over $\mathbb{F}_q$. They have an optimal bilinear complexity when $n\leq\frac{1}{2}q+1$, using evaluation on rational places only. However, the evaluations can be done at places of higher degrees of $\mathbb F_q(x)$, and we can construct an algorithm for extensions of any degree. In this section, we propose a recursive generic construction of Chudnovsky-type algorithms specialized to the projective line for the multiplication in all extensions $\mathbb{F}_{q^n}$, using places of increasing degrees. First, we consider the algorithm without derivative evaluations. Let $Q$ be a place of degree $n$. According to Corollary \ref{coroalgosansED}, we need to construct a set $\mathcal{P}$, containing places such that the sum of their degrees is equal to $2n-1$. We can assume $P_0$ and $P_\infty$ are always in $\mathcal{P}$. With such a set, Proposition \ref{algoavecPinfty} gives an algorithm $\mathcal{U}^\mathcal{P}_{q,n}(Q)$ for the multiplication in any extension. But at this step, we still do not have any information about how to compute the multiplications of the evaluations at places of an arbitrary degree. Concretely, let $P_i\in\mathcal{P}$ be a place of degree $d_i$. Then, for $f,g\in\mathcal{L}((n-1)P_\infty)$, the evaluations $f(P_i)$ and $g(P_i)$ are some elements in $\mathbb{F}_{q^{d_i}}$. To compute $(fg)(P_i)=f(P_i)g(P_i)$, we use the algorithm $\mathcal{U}_{q,{d_i}}^{\mathcal{P}_i}(P_i)$, where $\mathcal P_i$ is a set of places such that the sum of their degrees is equal to $2d_i-1$.
Such an algorithm is called a recursive Chudnovsky-type algorithm over the projective line.

\begin{definition}
Let $q$ be a prime power and $n > \frac{1}{2}q+1$ be a positive integer. We call a recursive Chudnovsky-type algorithm over the projective line an algorithm from Proposition \ref{algoavecPinfty}, that computes the multiplications in intermediate extensions with recursively-defined algorithms.
\end{definition}

\begin{example}\label{exx}
We consider the multiplication in $\mathbb{F}_{3^6}$. The function field $\mathbb{F}_3(x)$ has 4 rational places, 3 places of degree 2 and 8 places of degree 3. We denote by $P_\infty$ the place at infinity, and by $P_0$, $P_1$ and $P_2$ the three other rational places. The places of degree 2 are $P_1^2$, $P_2^2$ and $P_3^2$, and $P^3$ is one of the places of degree 3. Moreover, let $Q$ be a place of degree 6.
Then, a set of places containing the 4 rational places, 2 places of degree 2 and one of degree 3 is suitable, since we have $4\times1+2\times2+1\times3=11=2\times6-1$. For example, we can take $\mathcal{P}=\{P_\infty,P_0,P_1,P_2,P_1^2,P_2^2,P^3\}$. The products of the evaluations on rational places are multiplications in $\mathbb{F}_q$. The evaluations at places $P_i^2$ of degree 2 can be multiplied with 3 bilinear multiplications by $\mathcal{U}^{\mathcal{P}_2}_{3,2}$, defined in Remark \ref{quadchud}. We also need to construct the algorithm to multiply in the extension of degree 3. 

To do this, set $\mathcal{P}_3=\{P_0,P_1,P_\infty,P_1^2\}$. The sum of the degrees of the places in $\mathcal{P}_{3}$ is equal to $3\times1+1\times2=5=2\times3-1$, and this set is suitable. The algorithm $\mathcal{U}^{\mathcal{P}_3}_{3,3}$ computes 3 bilinear multiplications in $\mathbb{F}_q$, and $\mathcal{U}^{\mathcal{P}_2}_{3,2}(P_1^2)$ computing itself 3 more bilinear multiplications. Finally, its bilinear complexity is $\mu(\mathcal{U}^{\mathcal{P}_3}_{3,3})=3\times1+1\times3=6$. That is the best-known (and optimal) bound for the multiplication in the extensions of degree 3 of $\mathbb{F}_3$. Table \ref{shema3} illustrates the structure of $\mathcal{U}^{\mathcal{P}_3}_{3,3}$.
\begin{table}
\begin{center}
\begin{tikzpicture}
\node {${\mathcal U}^{\mathcal{P}_3}_{3,3}(P^3)$} [grow'=right,level distance=2.5cm,sibling distance=0.4cm]
  child {node {$P_0$}}
  child {node {$P_1$}}
  child {node {$P_\infty$}}
  child {node {$\mathcal{U}^{\mathcal{P}_2}_{3,2}(P_1^2)$}
         child {node {$P_0$}}
         child {node {$P_1$}}
         child {node {$P_\infty$}}
        }
  ;
\end{tikzpicture}
\vspace{.5em}
\caption{Diagram of the construction of $\mathcal{U}^{\mathcal{P}_3}_{3,3}(Q)$.}\label{shema3}
\end{center}
\end{table}

Back to the multiplication in extension of degree 6, the algorithm ${\mathcal U}^\mathcal{P}_{3,6}(Q)$ computes the multiplications of the evaluations on $P_\infty$, $P_0$, $P_1$, $P_2$, and use $\mathcal{U}^{\mathcal{P}_2}_{3,2}(P_1^2)$, $\mathcal{U}^{\mathcal{P}_2}_{3,2}(P_2^2)$ and $\mathcal{U}^{\mathcal{P}_3}_{3,3}(P^3)$. This gives the bilinear complexity $\mu({\mathcal U}^{\mathcal{P}}_{3,6})=4\times1+2\times3+1\times6=16$.
The tree on the left side of Table \ref{shema} gives an illustration of the algorithm.
\end{example}
\clearpage

\begin{table}
\begin{tabular}{c c}
\begin{tikzpicture}
\node  {${\mathcal U}^\mathcal{P}_{3,6}(Q)$}[grow'=right,level distance=1.5cm,sibling distance=0.5cm]
  child[yshift=2cm] {node {$P_0$}}
  child[yshift=1.8cm] {node {$P_1$}}  
  child[yshift=1.6cm] {node {$P_2$}}
  child[yshift=1.4cm] {node {$P_\infty$}}
  child[yshift=0.7cm] {node {$\mathcal{U}^{\mathcal{P}_2}_{3,2}(P_1^2)$}[sibling distance=0.5cm]
         child {node {$P_0$}}
         child {node {$P_1$}}
         child {node {$P_\infty$}}
         }
  child[yshift=-0.2cm] {node {$\mathcal{U}_{3,2}^{\mathcal{P}_2}(P_2^2)$}[sibling distance=0.5cm]
         child {node {$P_0$}}
         child {node {$P_1$}}
         child {node {$P_\infty$}}     
        }
    child[yshift=-1.4cm] {node {$\mathcal{U}_{3,3}^{\mathcal{P}_3}(P^3)$} [grow'=right,level distance=1.5cm,sibling distance=0.5cm]
        child {node {$P_0$}}
        child {node {$P_1$}}
        child {node {$P_\infty$}}
        child [yshift=0cm]{node {$\mathcal{U}_{3,2}^{\mathcal{P}_2}(P_1^2)$}[sibling distance=0.5cm]
            child {node {$P_0$}}
            child {node {$P_1$}}
            child {node {$P_\infty$}}
        }}
  ;
\end{tikzpicture}
& 
\begin{tikzpicture}
\node {$\mathcal{U}_{3,6}^{\mathcal{P}',\underline{u}}(Q)$} [grow'=right,sibling distance=0.5cm]
  child[yshift=1.5cm] {node {$2P_0$}
    child {node {$P_0$}}
    child[sibling distance=1cm] {node {$P_0'$}[sibling distance=0.5cm]
        child {node {$f'(P_0)g(P_0)$}}
        child {node {$f(P_0)g'(P_0)$}}}
  }
  child[yshift=1.2cm] {node {$P_1$}}  
  child[yshift=0.9cm] {node {$P_2$}}
  child[yshift=0.5cm] {node {$P_\infty$}}
  child[yshift=0cm] {node {$\mathcal{U}_{3,2}^{\mathcal{P}_2}(P_1^2)$}[sibling distance=0.5cm]
         child {node {$P_0$}}
         child {node {$P_1$}}
         child {node {$P_\infty$}}
         }
  child[yshift=-1cm] {node {$\mathcal{U}_{3,2}^{\mathcal{P}_2}(P_2^2)$}[sibling distance=0.5cm]
         child {node {$P_0$}}
         child {node {$P_1$}}
         child {node {$P_\infty$}}     
        }
  child[yshift=-2.2cm] {node {$\mathcal{U}_{3,2}^{\mathcal{P}_2}(P_3^2)$}[sibling distance=0.5cm]
         child {node {$P_0$}}
         child {node {$P_1$}}
         child {node {$P_\infty$}}     
        }
  ;
\end{tikzpicture}\end{tabular}
\vspace{.5em}
\caption{Diagram of the construction of $\mathcal{U}_{3,6}^{\mathcal{P}}(Q)$ and $\mathcal{U}_{3,6}^{\mathcal{P}',\underline{u}}(Q)$.}\label{shema}
\end{table}

\noindent Our strategy can be summarized as follows:

\begin{center}
 \rule{\linewidth}{1pt}
 \end{center}
\textbf{RPGC: Recursive Polynomial Generic Construction}

\noindent For $q$ a prime power and $n\geq 2$ a positive integer, let $Q$ be a place of degree $n$ of $\mathbb{F}_q(x)$. Then, $\mathcal{U}_{q,n}^{\mathcal{P}}(Q)$ is an algorithm for the multiplication in $\mathbb{F}_{q^n}$, with the following settings:
\begin{itemize}
    \item $\D=(n-1)P_\infty$,
    \item $\mathcal{P}=\{P_1,\ldots,P_N\}$ is a set of places such that $\sum_{i=1}^N\deg P_i=2n-1 $,
    \item the basis of $\mathcal{L}(\D)$ is $\{1,x,\ldots,x^{n-1}\}$,
    \item the basis of $\mathcal{L}(2\D)$ is $\{1,x,\ldots,x^{2n-1}\}$, and
    \item apply recursively RPGC to every non-rational places in $\mathcal{P}$.
\end{itemize}
\begin{center}
 \rule{\linewidth}{1pt}
 \end{center}

\begin{proposition}\label{genericalgo}
Let $q$ be a prime power and $n\geq 2$ an integer. RPGC is a set-up for a recursive Chudnovsky-type algorithm $\mathcal{U}_{q,n}^\mathcal{P}(Q)$ over the projective line for the multiplication in $\mathbb{F}_{q^n}$ that interpolates over polynomials. Its bilinear complexity is equal to 
$$\mu(\mathcal{U}^{\mathcal{P}}_{q,n})=\sum_{k}N_k\mu(\mathcal{U}^{\mathcal{P}_k}_{q,k}). $$  
where $N_k$ is the number of places of degree $k$ in $\mathcal{P}$, and $\mathcal{U}^{\mathcal{P}_k}_{q,k}$ is built with RPGC.
\end{proposition}

\begin{proof}
Let $Q$ be a degree $n$ place of $\mathbb{F}_q(x)$. We set once again $\D=(n-1)P_\infty$. Let $\mathcal{P}$ be a set of places of $\mathbb{F}_q(x)$, such that the sum of their degrees is equal to $2n-1$. We denote by $N_k$ the number of places of degree $k$ in $\mathcal{P}$. We can write $\mathcal{P}=\{P_1,\ldots,P_{N_1},P_1^2,\ldots,P_{N_2}^2,\ldots,P_1^d,\ldots,P_{N_d}^d\}$, where the $P_1,\ldots,P_{N_1}$ are rational places and for all $j=2,\ldots,d$, $P_i^j$ is a place of degree $j$ of $\mathbb{F}_q(x)$. We suppose that $Q$ is not in $\mathcal{P}$. In particular, the condition of bijectivity is given by $\sum kN_k=2n-1$, where $d$ is running over the degrees of all places in $\mathcal{P}$: the divisor $2\D-\sum_{P\in\mathcal{P}}P$ is of degree -1 and non-special, because of degree greater than $2g-2$, so zero-dimensional. Proposition \ref{algoavecPinfty} gives the main algorithm $\mathcal{U}_{q,n}^\mathcal{P}(Q)$.

Then, let $\mathcal{P}_2,\ldots,\mathcal{P}_d$ be sets of places such that for all $2\leq k\leq d$, we have $\sum_{P\in\mathcal{P}_k}\deg P=2k-1$.
We construct $\mathcal{U}_{q,k}^{\mathcal{P}_k}$ to multiply the evaluations at the places $P_i^k$ of any degree $k$. For $f,g\in\mathcal{L}(\mathcal{D})$, each product $f(P_i^k)g(P_i^k)$ is computed using $\mathcal{U}_{q,k}^{\mathcal{P}_k}(P_i^k)$. In particular, the algorithm $\mathcal{U}_{q,1}$ is just a bilinear multiplication in $\mathbb{F}_q$. Finally, the bilinear complexity given by the construction is 
\begin{center}
$\mu(\mathcal{U}^{\mathcal{P}}_{q,n})=\sum_{k}N_k\mu(\mathcal{U}^{\mathcal{P}_k}_{q,k}). $   
\end{center}\end{proof}

\begin{remark}
The whole construction is fully defined by $\mathcal{P},\mathcal{P}_2,\ldots,\mathcal{P}_d$ and $Q$.
\end{remark}

At this point, we do not know how to construct the set of places $\mathcal P$. Finding the set of places that gives the best bilinear complexity is a difficult task when $n$ is large. Hence, we focus on some families of compatible set to study the bilinear complexity of the recursive Chudnovsky-type algorithms over the projective line.

\subsection{Construct $\mathcal{P}$ taking places by increasing degrees}\label{degreecroissant}
A natural strategy is to construct the set $\mathcal{P}$ by taking every place by growing degrees until the sum of their degrees is equal to $2n-1$. If the sum is bigger than $2n-1$, we remove a place of appropriate degree. Such a set is denoted by $\mathcal{P}^{\text{deg}}$. 

\begin{definition}
We denote by $\mathcal{U}^{\mathcal{P}^{\deg}}_{q,n}$ a recursive Chudnovsky-type algorithm on the projective line for the multiplication in $\mathbb{F}_{q^n}$, where $\mathcal{P}^{\deg}$ is built taking places of increasing degrees. 
\end{definition}

The algorithms of Example \ref{exx} were built using this method.
Note that this construction is deterministic assuming that the places of $\mathbb{F}_q(x)$ have a given order. 
We know that the number of rational places over $\mathbb{F}_{q^d}(x)$ is equal to $q^d+1$. This value is the sum of the $kB_k$, for $k$ dividing $d$ and $B_k=B_k(\mathbb{F}_{q}(x))$ (\cite{stic2}, (5.40)). Then, the number of places of degree $d$ of $\mathbb{F}_q(x)$ is equal to

\begin{equation}\label{nbplaces}
B_d=\frac{1}{d}\left(q^d+1-\underset{k\neq d}{\sum_{k\mid d}}kB_k\right).    
\end{equation} 
Thus, for a given $n$ we can compute $d$, the smallest integer such that $\sum_{k=1}^dkB_k\geq 2n+1$. Then, the number $N_k$ of places of degree $k$ in $\mathcal{P}^{\deg}$ is given by  
    $$
N_k = \left\{
    \begin{array}{ll}
        0 & \mbox{if } k>d,\\
        \lceil\frac{1}{d}(2n-1-\sum_{i=1}^{d-1}iB_i)\rceil  & \mbox{if } k=d, \\
        B_k-1 & \mbox{if } 0\neq k\equiv-(2n-1-\sum_{i=1}^{d-1}iB_i)\pmod{d}, \\
        B_k & \mbox{elsewhere,}\end{array}\right.
    $$
where $B_k$ denotes the number of places of degree $k$ of $\mathbb{F}_q(x)$.
This allows one to compute iteratively the bilinear complexity since $\mu(\mathcal{U}^{\mathcal{P}^{\deg}}_{q,n})=\sum_k N_k\mu(\mathcal{U}^{\mathcal{P}^{\deg}}_{q,k})$, where $\mathcal{P}^{\deg}$ is defined in accordance with $k$ or $n$.
Table \ref{Tablebilcomp} shows this bilinear complexity for the extensions of degree lower than 18 of $\mathbb{F}_q$, with $q=2,3,4.$ We underline the bilinear complexity when it equals the one given by Table 2 of \cite{survey}, and we denote by $^+$  when we beat this complexity.
{\setlength{\tabcolsep}{3.5pt}
\begin{table}\begin{center}
 \begin{tabular}{|l|*{17}{c|}}
	\hline
	$n$  &\  2 \  & \ 3 \ & \  4 \ & 5 & 6 & 7 & 8 & 9 & 10 & 11 & 12 & 13 & 14 & 15 & 16 & 17 & 18 \\
	 \hline
	  	 $\mu(\mathcal{U}^{\mathcal{P}^{\deg}}_{2,n})$ & $\underline{3}$ & $\underline{6}$ & 11 & 15 & 18 & 26 & 29 & 37 & 40 & 48 & 51 & 60 & 65 & 70 & 78 & 81 & 90 \\
	\hline
	 $ \mu(\mathcal{U}^{\mathcal{P}^{\deg}}_{3,n})$ & $\underline{3}$ & $\underline{6}$ & $\underline{9}$ & $12$ & $16$ & $\underline{19}$ & 24 & 28 & 31 & 36 & 40 & 43 & 48 & 52 & 55 & 60 & 64
 \\
	\hline
	$ \mu(\mathcal{U}^{\mathcal{P}^{\deg}}_{4,n})$ & $\underline{3}$ & $5^+$ & $\underline{8}$ & $\underline{11}$ & $\underline{14}$ & $\underline{17}$ & $\underline{20}$ & $\underline{23}$ & $\underline{27}$ & $\underline{30}$ & $\underline{33}$ & $\underline{37}$ & 40 & $43^+$ & 47 & $50^+$ & 53\\
	\hline
\end{tabular}
\vspace{.5em}
\caption{Bilinear complexity of $\mathcal{U}^{\mathcal{P}^{\deg}}_{q,n}$ in small extensions of $\mathbb{F}_2$, $\mathbb{F}_3$ and $\mathbb{F}_4$.}\label{Tablebilcomp}
\end{center}\end{table}    }

The previous strategy of construction is not optimal, for instance because the ratio $\frac{\mu(\mathcal{U}^{\mathcal{P}^{\deg}}_{q,n})}{n}$ is not increasing. 
\begin{example}
For $q=2$ and $n=82$, we need a set of places such that the sum of their degrees is equal to 163. In this case $\mathcal{P}^{\deg}$ contains all places of $\mathbb{F}_2(x)$ of degrees lower than 6, and 8 places of degree 7. The bilinear complexity obtained is hence $\mu(\mathcal{U}^{\mathcal{P}^{\deg}}_{2,82})=503$. However, one can check that $\frac{\mu(\mathcal{U}^{\mathcal{P}^{\deg}}_{2,7})}{7}>\frac{\mu(\mathcal{U}^{\mathcal{P}^{\deg}}_{2,8})}{8}.$ Hence, it is better to use 7 places of degree 8 than 8 places of degree 7. Let $\mathcal{P}$ be a set containing all the places of degree lower than or equal to 6, and 7 places of degree 8. The algorithm obtained using this set (with the same sub-algorithms) has bilinear complexity $\mu(\mathcal{U}^{\mathcal{P}}_{2,82})=499$.
\end{example}

\subsection{Derivative evaluations in RPGC}

A possibility to improve the bilinear complexity of the algorithms - and the sub-algorithms involved -, is to use derivative evaluations. For example, we can use derivative evaluations instead of one of the places of the highest degree in an algorithm $\mathcal{U}_{q,n}^{\mathcal{P}^{\deg}}$ from the previous section. An illustration of this process is given in the following example.

\begin{example}\label{suiteex}
We consider the algorithm of Example \ref{exx}. The algorithms $\mathcal{U}_{3,2}^{\mathcal{P}_2}$ and $\mathcal{U}_{3,3}^{\mathcal{P}_3}$ having optimal bilinear complexities, there is no need to use derivative evaluations in this situation. For the extension of degree 6, it is possible to take the last place of degree 2 instead of the place of degree 3, and to use a derivative evaluation. We get a new algorithm $\mathcal{U}_{3,6}^{\mathcal{P'},\underline{u}}$. This time, we use the set $\mathcal{P}'=\{P_\infty,P_0,P_1,P_2,P_1^2,P_2^2,P_3^2\}$, with the coefficients given by $\underline{u}=(2,1,\ldots,1)$, i.e. $u_0=2$ and $u_i=1$ elsewhere. We evaluate at $P_0$ with multiplicity 2, and without multiplicity on the other places. Since $x$ is the local uniformizer for the local expansion at $P_0$, the expansion is of the form $f(P_0)+xf'(P_0)+\cdots$, for $f\in\mathcal{L}(5P_\infty)$. The truncated product of the evaluations of $f,g\in\mathcal{L}(5P_\infty)$ is given by $f(P_0)g(P_0)+x(f'(P_0)g(P_0)+f(P_0)g'(P_0))$, and requires 3 bilinear multiplications. The new algorithm ${\mathcal U}_{3,6}^{\mathcal{P}',\underline{u}}(Q)$ is obtained following Proposition \ref{algoavecPinfty}. Its bilinear complexity is $\mu({\mathcal U}_{3,6}^{\mathcal{P}',\underline{u}})=1\times3+3\times1+3\times3=15$, for one evaluation with multiplicity 2 on $P_0$, 3 evaluations on rational places and three at places of degree 2. This is the best-known bound (\cite{survey}, Table 2). This construction is illustrated in Table \ref{shema}, where $2P_0$ means that we evaluate at $P_0$ with multiplicity 2 and $P_0'$ is the second coefficient of the local expansion at $P_0$.

\end{example}

Table \ref{Tablebilcompderiv} shows some improvements of the bilinear complexities of $\mathcal{U}_{q,n}^{\mathcal{P}^{\deg}}$ given in Table \ref{Tablebilcomp}, when some derivative evaluations on rational places are used, as in Example \ref{suiteex}.

{\setlength{\tabcolsep}{3.5pt}
\begin{table}
\begin{center}
 \begin{tabular}{|l|*{17}{c|}}
	\hline
	$n$  &\  2 \  & \ 3 \ & \  4 \ & 5 & 6 & 7 & 8 & 9 & 10 & 11 & 12 & 13 & 14 & 15 & 16 & 17 & 18 \\
	 \hline
 	 $\mu(\mathcal{U}^{\mathcal{P'},\underline{u}}_{2,n})$ & - & - & 10 & 14 & - & $\underline{22}$ & 28 & 32 & 38 & 42 & 48 & 52 & 58 & 64 & 68 & 76 & 80 \\
	\hline
	 $ \mu(\mathcal{U}^{\mathcal{P'},\underline{u}}_{3,n})$ & - & - & - & - &  $\underline{15}$ & - & 23 & 27 & - & 35 & 39 & - &  47 & 51 & - & 59 & 63 \\
	\hline
	$ \mu(\mathcal{U}^{\mathcal{P'},\underline{u}}_{4,n})$ & - & - & - & - & - & - & - & - & - & - & - & - & - & - & - & - & -\\
	\hline
\end{tabular}
\vspace{.5em}

\caption{Some improvements of Table \ref{Tablebilcomp} thanks to derivative evaluations.}\label{Tablebilcompderiv}
\end{center}\end{table}    
}

\section{Asymptotical study for RPGC}

\subsection{An explicit construction for the asymptotical study}

Hitherto, the bilinear complexity obtained with RPGC can be computed for a fixed $n$, but we do not have an estimate relatively to $n$ yet. 
%In this purpose, we consider another strategy to construct the set of places.
First, note that given a Chudnovsky-type algorithm $\mathcal{U}_{q,n}^{\mathcal{P}_n}$ over the projective line for the multiplication in $\mathbb{F}_{q^n}$, we can define an algorithm of the same bilinear complexity for an extension of degree lower than $n$.

\begin{lemma}\label{majojo}
Let $\mathcal{U}_{q,n}^{\mathcal{P}_n}$ be a Chudnovsky-type algorithm over the projective line for the multiplication in $\mathbb{F}_{q^n}.$ Then, for all $m<n$, there exists a Chudnovsky-type algorithm over the projective line $\mathcal{U}_{q,m}^{\mathcal{P}_m}$ for the multiplication in $\mathbb{F}_{q^m}$ such that
$$\mu(\mathcal{U}_{q,m}^{\mathcal{P}_m})\leq\mu(\mathcal{U}_{q,n}^{\mathcal{P}_n}).$$
\end{lemma}

\begin{proof}
Let $\mathcal{U}_{q,n}^{\mathcal{P}_n}$ be a Chudnovsky-type algorithm on the projective line for the multiplication in $\mathbb{F}_{q^n}$. We have $\sum_{P\in\mathcal{P}_n}\deg P = 2n-1>2m-1$. Set $\mathcal{P}_m=\mathcal{P}_n$. Without loss of generality, we can assume that the place at infinity is not in $\mathcal{P}_m$ (in fact, if $P_\infty\in\mathcal{P}_n$, we can consider $\mathcal{P}_m=\mathcal{P}_n\setminus\{P_\infty\}$, and $\sum_{P\in\mathcal{P}_m}\deg P = 2n-2>2m-1 $).
Then, we build the algorithm $\mathcal{U}_{q,m}^{\mathcal{P}_m}$ from Theorem \ref{Algogene}, with $Q$ a place of degree $m$ of $\mathbb{F}_q(x)$ and $\D=(m-1)P_\infty$. The map $E_Q$ gives an isomorphism between $\mathcal{L}(\D)$ and $\mathbb{F}_{q^m}$. Since $\sum_{P\in\mathcal{P}_m}\deg P>2m-1$, the evaluation map $Ev_{\mathcal{P}_m}:\mathcal{L}(2\D)\longrightarrow\mathbb{F}_q^N$ is injective.
The conditions of Theorem \ref{Algogene} are verified, and we obtain an algorithm $\mathcal{U}_{q,m}^{\mathcal{P}_m}$. Moreover, if we compute the multiplications of the evaluations at the places in $\mathcal{P}_m$ as in $\mathcal{U}_{q,n}^{\mathcal{P}_n}$, we obtain that $\mu(\mathcal{U}_{q,m}^{\mathcal{P}_m})\leq\mu(\mathcal{U}_{q,n}^{\mathcal{P}_n}).$

\end{proof}

Let us consider a convenient construction for the asymptotical study. For any integer $l$, recall that the places of $\mathbb{F}_q(x)$ of degree dividing $l$ correspond to the $q^l+1$ rational places of $\mathbb{F}_{q^l}(x)$, and that $B_k=B_k(\mathbb{F}_q(x))$ denotes the number of places of degree $k$ of $\mathbb{F}_q(x)$.
Let $d$ be the smallest integer $d$ such that $q^d\geq2n$, i.e. such that $d-1\leq\log_q(2n)\leq d$. Then, the sum $\sum_{k\mid d}kB_k=q^{d}+1$ is greater than $2n-1$.  We construct the set $\mathcal{P}$ in two steps. First, we include only places of degrees dividing $d$ by increasing degrees, while the sum of their degrees is lower than $2n-1$. In anticipation of a future calculation, we do not include one of the rational places. Then, the second step is to adjust the set to obtain a sum exactly equals to $2n-1$.
More precisely, consider
\begin{equation}\label{S}
S=(\sum_{k\mid d \atop k\neq d}kB_k) -1.
\end{equation}
The quantity $S$ is the sum of the degrees of the places over all places of degrees strictly dividing $d$, minus a rational place that we does not count in the calculation. Now, consider $\delta=2n-1-S \pmod{d}$. If $\delta=0$, then $\sum_{P\in\mathcal{P}} \deg P =2n-1$ after the first step, and the algorithm is done. 
Elsewhere, we have to increase this sum by $\delta$. If $\delta\nmid d$, then we add a place of degree $\delta$ in $\mathcal{P^\text{div}}$. If $\delta\mid d$, the places of degree $\delta$ are already in $\mathcal{P}$. For $\ell$ the largest divisor of $d$ non equal to $d$, we adjust $\mathcal{P}$ by adding a place of degree $\ell+\delta$ and by removing a place of degree $\ell$.
The correctness of this construction will be proven in Proposition \ref{lemmePdiv}.
The set obtained is denoted by $\mathcal{P}^\text{div}$, and its construction is formalized in the following algorithm.  

\begin{algorithm}[H]
\caption{Construction of $\mathcal{P}^{\text{div}}$}\label{Pdiv}
\begin{algorithmic}
\Require $q,~ n>\frac{1}{2}q+1$. 
\Ensure  $\mathcal{P}^{\text{div}}$.
        \begin{enumerate}
          
          \item[(i)] Let $d$ be the smallest integer greater than $\log_q(2n)$, set $S=(\sum_{k\mid d \atop k\neq d}kB_k) -1$, and let $\ell$ be the greatest non-trivial divisor of $d$. Start to construct $\mathcal{P}^{\text{div}}$ including $N_k$ places of degree $k$, for all $k$ dividing $d$, with $N_k$ as follows:
    $$
N_k = \left\{
    \begin{array}{ll}
        \lfloor \frac{1}{d}(2n-1-{S})\rfloor & \mbox{if } k=d, \\
        q & \mbox{if } k=1, \\
        B_k & \mbox{if } k\mid d \mbox{ and } k\neq 1,d, \\
        0 & \mbox{elsewhere};\end{array}\right.
    $$
        \item[(ii)]  Consider $\delta=2n-1-S\pmod{d}$. Then,
        \begin{itemize}
            \item If $\delta=0$, the algorithm is done;
            \item If $\delta\nmid d$, add to $\mathcal{P}^{\text{div}}$ a place of degree $\delta$;
            \item If $0\neq\delta\mid d$ add to $\mathcal{P}^{\text{div}}$ a place of degree $\ell+\delta$ and remove a place of degree $\ell$.
        \end{itemize}
        \end{enumerate}
\end{algorithmic}
\end{algorithm}
\begin{definition}
We denote by $\mathcal{U}^{\mathcal{P}^{\text{div}}}_{q,n}$ a recursive Chudnovsky-type algorithm on the projective line for the multiplication in $\mathbb{F}_{q^n}$, with $\mathcal{P}^{\text{div}}$ given by Algorithm \ref{Pdiv}.
\end{definition}

\begin{remark}
For the asymptotical study, we can construct RPGC Chudnovsky-type algorithms with $\mathcal{P}^{\text{div}}$, and use the algorithms from the previous section in the recursion. For instance, we can use the bounds of Table \ref{Tablebilcompderiv}. 
\end{remark}

\begin{proposition}\label{lemmePdiv}
Let $q$ be a prime power, $n$ be an integer greater than $\frac{1}{2}q+1$ and $Q$ a place of degree $n$ of $\mathbb F_q(x)$. The recursive Chudnovsky-type algorithm over the projective line $\mathcal{U}_{q,n}^{\mathcal{P}^{\text{div}}}(Q)$ has bilinear complexity 
$$\mu(\mathcal{U}^{\mathcal{P}^{\text{div}}}_{q,n})	\leq \sum_{k\mid d}\frac{q^k}{k}\mu(\mathcal{U}^{\mathcal{P}_k}_{q,k}),$$
where $d$ is the smallest integer greater than $\log_q(2n),$ and $\mathcal{U}^{\mathcal{P}_k}_{q,k}$ is a Chudnovsky-type algorithm on the projective line for the multiplication in $\mathbb{F}_{q^k}$.
\end{proposition}
\begin{proof}
We have to check that $\mathcal{P}^{\text{div}}$ is well defined.
After the first step, $\mathcal{P}^{\text{div}}$ contains only places of degree dividing $d$. Consider the quantities $\delta$, $S$ and $\ell$ from Algorithm \ref{Pdiv}. If $\delta=0$ the algorithm is done. If $\delta \nmid d$, we can add a place of degree $\delta$ in $\mathcal{P}^{\text{div}}$. If $0\neq\delta \mid d$, we have to argue differently since the places of degree $\delta$ are already in $\mathcal{P}^{\text{div}}$. We can always add a place of degree $\ell+\delta$ since either it does not divide $d$ or it is equal to $d$.  
In fact, $\ell$ is lower than or equal to $d/2$ and $\delta\leq d/2$ as well. Hence, $\ell<\ell+\delta\leq d$. 
Suppose that $\ell+\delta\neq d$. Since $\ell$ is the greatest non-trivial divisor of $d$, $\ell+\delta$ does not divide $d$, and we can add a place of degree $\ell+\delta$ in $\mathcal{P}^{\text{div}}$. 
If $\delta=\ell=d/2$, then $\delta+\ell=d$, and we have to verify that the number $N_d$ of places of degree $d$ used after the first step is strictly lower than $B_d$, so that one can add a place of degree $d$.
Let us show that $N_d<B_d$. After step $(i)$, the sum of the degrees of the places in $\mathcal{P}^{\text{div}}$ is equal to $\sum_{k \mid d}dN_d=S+d\lfloor \frac{1}{d}(2n-1-{S})\rfloor\leq 2n-1$.
Suppose that $N_d=B_d$, i.e. all places of degree $d$ are included after step $(i)$. The condition $d\geq\log_q(2n)$ attests that $\sum_{k\mid d}kB_k=q^d+1$ is greater than $2n+1$. 
After step $(i)$, the sum of the degrees of the places is hence greater than $2n$, since we do not count a rational place.
But this sum has to be lower than or equal to $2n-1$, which gives a contradiction.
Finally, the number $N_d$ of places of degree $d$ after the first step is
\begin{equation}\label{NdBd}
N_d=\lfloor\frac{1}{d}(2n-1-{S})\rfloor<B_d,
\end{equation}
and we can add a place of degree $d$. 
Now, Algorithm \ref{Pdiv} is complete, and we denote by $N_k^{\text{div}}$ the number of places of degree $k$ in $\mathcal{P}^{\text{div}}$ after $(i)$ and $(ii)$. It remains to verify that the sum of the $kN_k^{\text{div}}$ is equal to $2n-1$. Thanks to the euclidean division of $2n-1-{S}$ by $d$, we have 
$$\sum_{P\in\mathcal{P}^{\text{div}}}\deg P=\sum kN_k^{\text{div}}=S+\delta+d\lfloor\frac{1}{d}(2n-1-{S})\rfloor=2n-1.$$

Hence, $\mathcal{P}^{\text{div}}$ is well defined, and we obtain a recursive Chudnovsky-type algorithm over the projective line. Its bilinear complexity verifies
$$\mu(\mathcal{U}^{\mathcal{P}^{\text{div}}}_{q,n})=\sum_{k}N_k^{\text{div}}\mu(\mathcal{U}^{\mathcal{P}_k}_{q,k})=\sum_{k\mid d}N_k^{\text{div}}\mu(\mathcal{U}^{\mathcal{P}_k}_{q,k})+\Delta,$$
where $\Delta= \left\{
    \begin{array}{ll}
    \mu(\mathcal{U}^{\mathcal{P}_\delta}_{q,\delta}) & \text{if }\delta\nmid d, \\ 
    \mu(\mathcal{U}^{\mathcal{P}_{\ell+\delta}}_{q,\ell+\delta}) & \text{if } 0\neq\delta\mid d\text{ and }\ell+\delta\neq d,
    \\  0 & \text{if }\delta=0\text{ or }\ell+\delta=d.\end{array}\right. $

Recall that $\ell\leq\ell+\delta\leq d$. Thanks to Lemma \ref{majojo}, we define $\mathcal{U}^{\mathcal{P}_{\delta}}_{q,\delta}$ (resp. $\mathcal{U}^{\mathcal{P}_{\ell+\delta}}_{q,\ell+\delta}$) using ${\mathcal{P}_{\delta}}={\mathcal{P}_{d}}$ and ${\mathcal{P}_{\ell+\delta}}={\mathcal{P}_{d}}$. Hence, $\mu(\mathcal{U}^{\mathcal{P}_{d}}_{q,\delta})\leq\mu(\mathcal{U}^{{\mathcal{P}_{d}}}_{q,d})$,
and $\mu(\mathcal{U}^{{\mathcal{P}_{d}}}_{q,\ell+\delta})\leq\mu(\mathcal{U}^{{\mathcal{P}_{d}}}_{q,d})$.
If $\Delta\neq0$, the inequality (\ref{NdBd}) attests that $N_d^{\text{div}}<B_d$. 
Since $\Delta\leq\mu(\mathcal{U}^{\mathcal{P}_{d}}_{q,d})$, we have that $N_d^{\text{div}}\mu(\mathcal{U}^{\mathcal{P}_d}_{q,d})+\Delta\leq B_d\mu(\mathcal{U}^{\mathcal{P}_{d}}_{q,d})$. 
For all $k$, we have $N_k^{\text{div}}\leq B_k$. Moreover, for all $k\geq2$ we have $B_k\leq\frac{q^k}{k}$. Recalling that $N_1^{\text{div}}\leq q$, we obtain 
$$\mu(\mathcal{U}^{\mathcal{P}^{\text{div}}}_{q,n})	\leq \sum_{k\mid d}\frac{q^k}{k}\mu(\mathcal{U}^{\mathcal{P}_{k}}_{q,k}).$$

\end{proof}

 \subsection{Bound for the bilinear complexity of RPGC}

Our bound for the bilinear complexity requires to introduce the iterated logarithm.

\begin{definition}\label{logstar}
For all integer $n$, the iterated logarithm of $n$, denoted by $ \log^*(n)$,  is defined by the following recursive function:

$$
\log^*(n) = \left\{
    \begin{array}{ll}
        0 & \mbox{if } n \leq 1 \\
        1 + \log^*(\log(n)) & \mbox{elsewhere.}
    \end{array}
\right.
$$

\noindent
This value corresponds to the number of times the logarithm is iteratively applied from $n$ in order to obtain a result lower than or equal to 1. 
\end{definition}
To get a uniform bound, we have to specify the basis of the logarithm. Here, we will use $q$ and finally $\sqrt{q}$ as basis. When the basis of the logarithm is a real greater than $e^\frac{1}{e}$, this function is well defined. But the iterated logarithm can be defined with any basis $a$, for a real $a$ strictly greater than 1. Actually, we will deal with the case of the basis $\sqrt{2}$, which is between 1 and $e^\frac{1}{e}\simeq1,44467\ldots$.
For $1<a<e^\frac{1}{e}$, there exists $x_0>1$ such that $x_0= \sup \{x\mid x=\log_a(x)\}$. Hence, the values obtained by applying successively $log_a$ to $x\geq x_0$ converge to $x_0>1$, and the stopping step of the function defined above cannot be reached. We can define $\log_a^*$ for any $a$ between 1 and $e^\frac{1}{e}$ by changing the stopping step to $\log_a^*(n)=0$ if $n\leq\lfloor x_0\rfloor +1$. For $a={\sqrt{2}}$, this number is $4=\log_{\sqrt{2}}(4)$, and the iterated logarithm is given by
\begin{equation}\label{logstarsqrt2}
\log_{\sqrt{2}}^*(n) = \left\{
    \begin{array}{ll}
        0 & \mbox{if } n \leq 5 \\
        1 + \log_{\sqrt{2}}^*(\log_{\sqrt{2}}(n)) & \mbox{elsewhere.}
    \end{array}
\right.
\end{equation}
These are very slow-growing functions (\cite{babotu}, Table II). The study of the bound also requires the following lemma:

\begin{lemma}\label{lemma}
For all prime power $q\geq2$ and all integer $d\geq1$, we have $dq^{\frac{d}{2}+1}-q^{d}\leq q^{d+1}$.
\end{lemma}
\begin{proof}
Let $q$ be a prime power and $d\geq1$ an integer. The inequality above is equivalent to $d\leq q^{\frac{d-2}{2}}(q+1)$, and then to $\log_q(\frac{d}{q+1})\leq \frac{d-2}{2}.$
Hence we have to show that $d-2-2\log_q(d)+2\log_q(q+1)\geq0.$
We can write this inequality
$d\ln{q}-2\ln{q}-2\ln{d}+2\ln{(q+1)}\geq0$.
For all $d\in\mathopen{[}1\,;+\infty\mathclose{[}$, we define $$f_q(d)=d\ln{q}-2\ln{q}-2\ln{d}+2\ln{(q+1)}.$$
This function is derivable over $\mathopen{[}1\,;+\infty\mathclose{[}$ and its derivative is
$f_q'(d)=\ln{q}-\frac{2}{d}$. Hence $f_q$ has a minimum in $\frac{2}{\ln{q}}$. For $q\geq8$, $\frac{2}{\ln{q}}\leq1$, so $f_q$ is growing over $\mathopen{[}1\,;+\infty\mathclose{[}$ and its minimum is given by
$$f_q(1)=\ln{q}-2\ln{q}+2\ln{(q+1)}\geq 0.$$
For $q\leq7$, $f_q$ reaches its minimum in $\frac{2}{\ln{q}}>1$ and it can be checked that these minima are positive.
Then, the inequality of the lemma is holding for all prime power $q$ and all $d\geq1. $
\end{proof}
We can finally state the following result:
\begin{theorem}\label{g0}
Let $q$ be a prime power and $n\geq2$ a positive integer. Then, there exists a recursive Chudnovsky-type algorithm $\mathcal{U}^{\mathcal{P}}_{q,n}$ over the projective line for the multiplication in $\mathbb{F}_{q^n}$ over $\mathbb{F}_q$ with a uniform upper bound for its bilinear complexity
$$\mu(\mathcal{U}^{\mathcal{P}}_{q,n})\leq Cn\left(\frac{4q^2}{(q-1)}\right)^{\log_{\sqrt{q}}^*(2n)},$$
where $C=1$ for $q\geq3$ and $C=\frac{14}{5}$ for $q=2$.
\end{theorem}

\begin{proof}\label{bilboundproof}
When $n\leq\frac{1}{2}q+1$, the construction of section \ref{smallext} showed we can construct a Chudnovsky-type algorithm of optimal bilinear complexity over the project line, and the bound is verified. For $n>\frac{1}{2}q+1$, We construct the set $\mathcal{P}=\mathcal{P}^{\text{div}}$ with Algorithm \ref{Pdiv} and use RPGC. Then, Lemma \ref{majojo} and Proposition \ref{lemmePdiv} give

\[\begin{aligned}
\mu(\mathcal{U}^{\mathcal{P}^{\text{div}}}_{q,n})	&\leq \sum_{k\mid d_1}\frac{q^k}{k}\mu(\mathcal{U}^{\mathcal{P}^{\text{div}}}_{q,k})
\\		&\leq \left(\frac{q^{d_1}}{d_1}+\sum_{k=1}^{d_1/2}\frac{q^k}{k}\right)\mu(\mathcal{U}^{\mathcal{P}^{\text{div}}}_{q,d_1})
\\			&\leq \left(\frac{q^{d_1}}{d_1}+\frac{q^{\frac{d_1}{2}+1}-1}{q-1}\right)\mu(\mathcal{U}^{\mathcal{P}^{\text{div}}}_{q,d_1})
\\          &= \left(\frac{q^{d_1+1}-q^{d_1}+d_1q^{\frac{d_1}{2}+1}-d_1}{d_1(q-1)}\right)\mu(\mathcal{U}^{\mathcal{P}^{\text{div}}}_{q,d_1}).
\end{aligned}\]
Thanks to Lemma \ref{lemma}, for all $q\geq2$ and all $d_1\geq1$, we have $d_1q^{\frac{d_1}{2}+1}-q^{d_1}\leq q^{d_1+1}$. Hence,
\[\begin{aligned}
\mu(\mathcal{U}^{\mathcal{P}^{\text{div}}}_{q,n})	&\leq \left(\frac{2q^{d_1+1}}{d_1(q-1)}\right)\mu(\mathcal{U}^{\mathcal{P}^{\text{div}}}_{q,d_1})
\\			&\leq \left(\frac{4q^2n}{d_1(q-1)}\right)\mu(\mathcal{U}^{\mathcal{P}^{\text{div}}}_{q,d_1}),
\end{aligned}\]
because $d_1$ is such that $d_1-1<\log_q(2n)\leq d_1$.

We can similarly give a bound for $\mu(\mathcal{U}_{q,d_1})$, using $d_2-1<\log_q(2d_1)\leq d_2$, and $\mathcal P ^{\text{div}}$ for the extension of degree $d_1$. It follows that
$$\mu(\mathcal{U}^{\mathcal{P}^{\text{div}}}_{q,d_1})\leq \left(\frac{4q^2d_1}{d_2(q-1)}\right)\mu(\mathcal{U}^{\mathcal{P}^{\text{div}}}_{q,d_2}).$$
and hence, 
$$\mu(\mathcal{U}^{\mathcal{P}^{\text{div}}}_{q,n})\leq \frac{n}{d_2}\left(\frac{4q^2}{(q-1)}\right)^2\mu(\mathcal{U}^{\mathcal{P}^{\text{div}}}_{q,d_2}).$$
Doing $i$ times this process, we get
$$\mu(\mathcal{U}^{\mathcal{P}^{\text{div}}}_{q,n})\leq \frac{n}{d_i}\left(\frac{4q^2}{(q-1)}\right)^i\mu(\mathcal{U}^{\mathcal{P}^{\text{div}}}_{q,d_i}),$$
with $d_j-1<\log_q(2d_{j-1})\leq d_j$, for all $j$ between 2 and $i$. We now have to find $i$ such that the recursion stops. We want $d_i<1$,
but $d_i>\log_q(2d_{i-1})$, then it suffices that $\log_q(2d_{i-1})<1$. Actually, for $2\leq j\leq i$, we have $\log_q(2d_{j-1})<d_j$. Hence, we are looking for $i$ such that 
$$\underbrace{\log_q(2\log_q(\ldots(2\log_q}_{i~terms}(2n))\ldots))<1.$$
Notice that for all $a\in\mathbb{R}_+^*$ and $q>1$, $\log_{\sqrt{q}}(a)=2\log_q(a)$. Then, 
$$\underbrace{\log_q(2\log_q(\ldots(2\log_q}_{i~terms}(2n))\ldots))=\frac{1}{2}\underbrace{\log_{\sqrt{q}}(\log_{\sqrt{q}}(\ldots(\log_{\sqrt{q}}}_{i~terms}(2n))\ldots)).$$
First, we consider $q\geq3$. By Definition \ref{logstar} of $\log^*$, $i=\log_{\sqrt{q}}^*(2n)$ is convenient. We finally get the bound

$$\mu(\mathcal{U}^{\mathcal{P}^{\text{div}}}_{q,n})\leq n\left(\frac{4q^2}{(q-1)}\right)^{\log_{\sqrt{q}}^*(2n)}\hbox{, for }q\geq3.$$

For $q=2$, we process similarly using the iterated logarithm with basis $\sqrt{2}$ as in \eqref{logstarsqrt2}, so we stop the recursion at $n\leq5$. We can use the algorithms given in Section 4 for these small extensions. From Table \ref{Tablebilcompderiv}, recall that $\mu(\mathcal{U}^{\mathcal{P},\underline{u}}_{2,5})=14$. For $q=2$ and $n\geq5$ we obtain the bound
$$\mu(\mathcal{U}^{\mathcal{P}^{\text{div}}}_{2,n})\leq n\frac{\mu(\mathcal{U}^{\mathcal{P}^{\deg}}_{2,5})}{5}\left(\frac{4q^2}{(q-1)}\right)^{\log^*_{\sqrt{2}}(2n)}=\frac{14}{5}n\left(\frac{4q^2}{(q-1)}\right)^{\log^*_{\sqrt{2}}(2n)}.$$
Moreover, we can check in Table \ref{Tablebilcomp} the bilinear complexity for the algorithms of multiplication in extensions of degree lower than 5. Tables \ref{Tablebilcomp} and \ref{Tablebilcompderiv} gives $\mu(\mathcal{U}^{\mathcal{P}^{\deg}}_{2,2})=3$, $\mu(\mathcal{U}^{\mathcal{P}^{\deg}}_{2,3})=6$ and $\mu(\mathcal{U}^{\mathcal{P},\underline{u}}_{2,4})=10$. These values verify the latest bound.$ $
\end{proof}

The given bound does not count all places in the construction, nor does it count the possible use of derivative evaluations, but it gives an information relatively to $n$. Moreover, the construction of the $\mathcal{P}^{\text{div}}$ is more restrictive than $\mathcal{P}^{\text{deg}}$, and we can think that $\mu(\mathcal{U}^{\mathcal{P}^{\deg}}_{q,n})\leq\mu(\mathcal{U}^{\mathcal{P}^{\text{div}}}_{q,n}),$ but this result is neither proven nor has a counter-example yet. 

\begin{proposition}
Asymptotically, the recursive Chudnovsky-type algorithm $\mathcal{U}^{\mathcal{P}^{\text{div}}}_{q,n}$ for the multiplication in $\mathbb{F}_{q^n}$ of Theorem \ref{g0} has bilinear complexity $$\mu(\mathcal{U}^{\mathcal{P}^{\text{div}}}_{q,n})\in\mathcal{O}(n(2q)^{\log_q^*(n)}).$$
\end{proposition}
\begin{proof}
In the proof of Theorem \ref{g0}, we consider that asymptotically the number of places of degree $d$ of $\mathbb{F}_q(x)$ is in $\mathcal{O}(\frac{q^d}{d})$. Then, we have that $\mu(\mathcal{U}^{\mathcal{P}^{\text{div}}}_{q,n})\in\mathcal O \left(\frac{q^{d_1}}{d_1}\mu(\mathcal{U}^{\mathcal{P}^{\text{div}}}_{q,d_1})\right)$. Remaining calculations are similar, and we obtain $\mu(\mathcal{U}^{\mathcal{P}^{\text{div}}}_{q,n})\in\mathcal{O}(n(2q)^{\log_q^*(n)})$. Details can be found in the proof of the elliptic case \cite{babotu}. 
\end{proof}

This asymptotic bound is the same as for the construction of \cite{babotu}, i.e. with algorithms constructed over elliptic curves. 
On some examples, when the extension is large enough and with a good choice of curve, the algorithm over an elliptic curve might have a lower bilinear complexity, because a function field of genus $1$ can have more places of a fixed degree than the projective line. However, it is not clear that such a generic construction can be obtained by using elliptic curves. Moreover, there is no uniform bound for the bilinear complexity of the construction with elliptic curves yet.

The generalized Karatsuba algorithm, which is a Divide and  Conquer construction based on the multiplication of two polynomials of degree 1 (i.e. $\mathcal U_{2,2}^{\mathcal{P}_2}$ here), requires $\mathcal{O}(n^{\log_23})$ multiplications, all of them bilinear. This is much more than the bilinear complexity obtained with our construction.

\subsection{Complexity of the construction of $\mathcal{U}_{q,n}^{\mathcal{P}^{\text{div}}}$}

We can now describe the complexity of the construction of our algorithms. More precisely, we focus on the family of algorithms $\mathcal{U}_{q,n}^{\mathcal{P}^{\text{div}}}$ obtained using the sets of places $\mathcal{P}^{\text{div}}$ (Algorithm \ref{Pdiv}) to estimate the asymptotic behaviour of the algebraic complexity of the construction. This complexity is given by the number of elementary operations in $\mathbb{F}_q$. We use the standard Landau notation $\mathcal{O}$. Our result requires the following lemma.

\begin{lemma}\label{lemmeconstru}
Let $q$ be a prime power and $n,d$ be positive integers. Let $P$ be a place of degree $d$ of $\mathbb{F}_q(x)$. The $d\times n$ matrix of the evaluations of $x^i$ at $P$, for $i=0,\ldots,n-1$, can be computed with $\mathcal{O}(dn)$ operations in $\mathbb{F}_q$.
\end{lemma}
\begin{proof}
Let $P(x)=\sum_{j=0}^{d}a_jx^j$ be the monic irreducible polynomial associated to the place $P$. For any positive integer $i$, the evaluation of $x^i$ at $P$ is by definition its class in the residue class field (Cf. \cite{stic2}), i.e. in $\frac{\mathbb{F}_q[x]}{(P(x))}= \mathbb{F}_{q^d}$. Let $M$ be the $d\times n$ matrix over $\mathbb{F}_q$ of the evaluations of $x^i$ at $P$, for $i=0,\ldots,n-1$. More precisely, the $i-$th column of $M$ is the evaluation of $x^{i-1}$ at $P$ in the basis $\{1,\alpha,\ldots,\alpha^{d-1}\}$ of $\mathbb{F}_{q^d}$, with $\alpha$ a root of $P(x)$.
We denote by $M_{k,l}$ the coefficient in the $k-$th line and in the $l-$th column of $M$. The $d$ first columns are given by the identity matrix. Then, the $d+1-$th column is given by the opposite of the coefficients of $P(x)$, i.e. $M_{k,d+1}=-a_{k-1}$, for $k=1,\ldots,d$, since $\alpha^d=-\sum_{j=0}^{d-1}a_j\alpha^j$.
Then, every column can be computed from the previous ones and $P(x)$ with $2d$ operations. Indeed, Let $l$ be strictly greater than $d+1$. Then,
$$\begin{array}{ll}
    \alpha^l & = \alpha\times \alpha^{l-1} \\
     & = \alpha\times \sum_{k=1}^d M_{k,l-1}\alpha^{k-1}  \\
     & = -M_{d,l-1}a_{k-1}+\sum_{k=2}^d (M_{k-1,l-1}-M_{d,l-1}a_{k-1})\alpha^{k-1}.
\end{array}$$
Hence $M_{1,l}=-M_{d,l-1}a_{k-1}$, and for $k\geq2$ the coefficient $M_{k,l}$ is given by $M_{k,l}=M_{k-1,l-1}-a_{k-1}M_{d,l-1}.$ Thus, it takes $2d$ operations to compute each of the $n-d-1$ columns. Finally, the matrix $M$ can be computed in time $\mathcal{O}(dn).$ 
\end{proof}

As first stated by Shparlinski, Tsfasman and Vl\v{a}du\c{t} \cite{shtsvl}, the most expensive part of constructing Chudnovsky-type algorithms is to find a degree $n$ place in the function field. In our case, this means constructing a monic irreducible polynomial of degree $n$ over $\mathbb{F}_q$, and there exist algorithms to do it in polynomial time.

\begin{theorem}\label{construccomp}
The recursive Chudnovsky-type algorithm over the projective line $\mathcal{U}^{\mathcal{P}^{\text{div}}}_{q,n}$ of Theorem \ref{g0} given by RPGC is constructible deterministically and in time $\mathcal{O}(n^4)$. 
\end{theorem}
\begin{proof}
To obtain the degree $n$ place of $\mathbb{F}_q(x)$, we compute a monic irreducible polynomial of degree $n$ over $\mathbb{F}_q$ thanks to Shoup \cite{Shoup2} in time $\mathcal{O}(n^4)$. We use as a divisor $\D=(n-1)P_\infty$, thus we set $\{1,x,\ldots,x^{n-1}\}$ as the basis of $\mathcal{L}(\D)$, and $\{1,x,\ldots,x^{2n-2}\}$ as basis of $\mathcal{L}(2\D)$.

From Theorem \ref{g0}, constructing the places of degrees dividing $d$, for $\log_q(2n) \leq d \leq \log_q(2n)+1$, is enough. Asymptotically, the basis of the logarithm makes no difference and we can consider that $\mathcal{O}(\frac{q^d}{d})$ places of degree $d$ are required, with $d\in\mathcal{O}(\log(2n))$. Hence, we have to construct the rational places of $\mathbb{F}_{q^d}(x)$ and group them in places of $\mathbb{F}_q(x)$ by applying them the iterated Frobenius. From Von zur Gathen and Gerhard book \cite{voge}, Algorithm 14.26 computes the iterated Frobenius, which for a given $\alpha\in\mathbb{F}_{q^d}$ outputs $\alpha,\alpha^q,\ldots,\alpha^{q^k}$, for $k$ dividing $d$. Its complexity is in time $\mathcal{O}(M(d)^2\log(d)\log(k))$, where $M(d)\in\mathcal{O}(d^\omega)$, with $\omega$ the best exponent for the multiplication of two matrices of size $d\times d$. Currently, $\omega\simeq2,373\ldots$. With this algorithm, the $\mathcal{O}(\frac{q^d}{d})$ places of degree $d$ can be constructed with $\mathcal{O}\big(\frac{q^d}{d}M(d)^2\log(d)\log(d)\big)=\mathcal{O}\big(2n(\log(2n))^{2\omega-1}\log(\log2n)^2\big)$ operations. 

Then, we construct the matrices of $\Tilde{E_P}$, $\Tilde{Ev_P}^{-1}$, $E_Q$ and $Ev_Q^{-1}$ of $\mathcal{U}_{q,n}^{\mathcal{P}^{\text{div}}}$. The matrix $Ev_Q^{-1}$ is the identity. The construction of $E_Q$ is given by the evaluation at $Q$ of the basis of $\mathcal{L}(2\D)$ and can be computed in $\mathcal{O}((2n)^2)$ operations thanks to Lemma \ref{lemmeconstru}. The matrix $\Tilde{Ev}_P$ is obtained with $\mathcal{O}(\frac{q^d}{d})$ matrices $d\times 2n-1$ as in Lemma \ref{lemmeconstru}. Hence, it can be computed in $\mathcal{O}(\frac{q^d}{d}2nd)=\mathcal{O}(4n^2)$ operations. All other required matrices are sub-matrices or inverses of $\Tilde{Ev}_P$. The inverses of the matrices can be done with Strassen algorithm \cite{stra2} and can be performed in $\mathcal{O}((2n)^{\log_2(7)})$. 

Finally, the algorithm $\mathcal{U}_{q,n}^{\mathcal{P}^{\text{div}}}$ involves recursively defined algorithms to compute multiplications in intermediate extensions, and we need to construct $\mathcal{U}_{q,d}^{\mathcal{P}^{\text{div}}}(Q_d)$ for $Q_d$ traversing the set of places used for the interpolation. Asymptotically, recall that we consider  $\mathcal{O}(\frac{q^d}{d})$ places of degree $d\in\mathcal{O}(\log(2n))$. 
For a given degree $d$ place $Q_d$, the construction $\mathcal{U}_{q,d}^{\mathcal{P}^{\text{div}}}(Q_d)$ is negligible, since $d\in\mathcal{O}(\log(2n))$.  
Nevertheless, the matrices $E_{Q_d}$ have to be computed for $\mathcal{O}(\frac{q^d}{d})$ places of degree $d$. Each of these can be obtained in $\mathcal{O}(d\log d)$ operations by Lemma \ref{lemmeconstru}, and the construction of all these matrices costs $\mathcal{O}(q^d\log d)=\mathcal{O}(2n\log \log 2n)$ operations in the base field.

Putting this all together, we obtain a complexity of construction in $\mathcal{O}(n^4)$ operations. \end{proof}

This bound is for a deterministic construction. In practice, we can construct the monic irreducible polynomial of degree $n$ using the Las Vegas-type algorithm by Couveignes and Lercier \cite{cole2}, with running time $n^{1+\epsilon(n)}\times(\log q)^{5+\epsilon(q)}$, where $\epsilon(x)$ are functions in $\mathcal{O}(1)$.

\begin{theorem}
Using a probabilistic algorithm to compute the monic irreducible polynomial of degree $n$, the recursive Chudnovsky-type algorithm over the projective line $\mathcal{U}^{\mathcal{P}^{\text{div}}}_{q,n}$ of Theorem \ref{g0} given by RPGC is constructible in expected time $\mathcal{O}((2n)^{\log_2 7})$. 
\end{theorem}
\begin{proof}
In the previous proof, we replace Shoup's algorithm by the one of Couveignes and Lercier \cite{cole2} to find the monic irreducible polynomial. Then, the heaviest part of the complexity consists in computing the inverse of the matrices. 
\end{proof}

\begin{remark}\label{conell}
The construction over elliptic curves \cite{babotu} is also polynomial but more expensive. For instance, the construction of the degree $n$ place requires to compute the iterated Frobenius of a nontrivial rational point of the curve over $\mathbb{F}_{q^n}$. With the algorithm from \cite{voge}, as in the proof of Theorem \ref{construccomp}, it takes $\mathcal{O}(n^{2\omega}\log(n)^2)$ operations to construct this degree $n$ place. This is already above the complexity of the construction of our algorithms. Moreover, there is no precise estimation of this complexity yet.\end{remark}

\section{Further works}
We studied the method of Chudnovsky and Chudnovsky for the genus $g=0$. This allowed us to introduce a strategy for constructing a polynomial interpolation algorithm having a competitive bilinear complexity for any extension. The total complexity of these algorithms still needs to be improved, in particular by studying their scalar complexity.

\bibliographystyle{plain}
\bibliography{biblio}

\end{document}